\documentclass[12pt]{amsart}

\usepackage{amsfonts, amsmath, amssymb, stmaryrd, color, enumerate, relsize, mathrsfs,array}
\usepackage[utf8]{inputenc}
\usepackage[pdftex]{graphicx}
\usepackage{xy}
\usepackage{float}

\usepackage[backref=true,
            style=alphabetic,
            doi=false,
            bibencoding=utf8,
            maxnames=99,
            maxalphanames=99]{biblatex}
\DeclareNameAlias{author}{family-given}
\usepackage[dvipsnames]{xcolor}
\usepackage{hyperref}
\hypersetup{colorlinks,linkcolor={blue},citecolor={blue},urlcolor={red}} 
\usepackage[capitalise]{cleveref}
\usepackage{tikz-cd}
\tikzcdset{
  cells={font=\everymath\expandafter{\the\everymath\displaystyle}},
}
\usepackage{scalerel, stackengine}
\usepackage[margin=1in,footskip=.5in]{geometry}
\usepackage{etoolbox} 
\makeatletter
\let\ams@starttoc\@starttoc
\makeatother
\usepackage[parfill]{parskip}
\makeatletter
\let\@starttoc\ams@starttoc
\patchcmd{\@starttoc}{\makeatletter}{\makeatletter\parskip\z@}{}{}
\makeatother

\usepackage{amsthm,aliascnt}
\usepackage{url}

\newcommand{\mynewtheorem}[2]{
  \newaliascnt{#1}{theorem}
  \newtheorem{#1}[#1]{#2}
  \aliascntresetthe{#1}
  \expandafter\def\csname#1autorefname\endcsname{#2}
}

\newtheorem{theorem}{Theorem}[section]
\mynewtheorem{thm}{Theorem}
\mynewtheorem{proposition}{Proposition}
\mynewtheorem{p}{Proposition}
\mynewtheorem{lemma}{Lemma}
\mynewtheorem{lem}{Lemma}
\mynewtheorem{corollary}{Corollary}
\mynewtheorem{cor}{Corollary}
\mynewtheorem{conjecture}{Conjecture}

\theoremstyle{definition}
\mynewtheorem{exercise}{Exercise}
\mynewtheorem{definition}{Definition}
\mynewtheorem{de}{Definition}
\mynewtheorem{example}{Example}
\mynewtheorem{exm}{Example}
\mynewtheorem{remark}{Remark}
\mynewtheorem{rk}{Remark}

\renewcommand*{\do}[1]{
\expandafter\newcommand \csname #1#1\endcsname{\mathbb{#1}}}
\docsvlist{B,C,D,E,F,G,H,I,J,K,L,M,N,O,P,Q,R,T,U,V,W,X,Y,Z}
\renewcommand{\AA}{\mathbb{A}}

\renewcommand*{\do}[1]{
\expandafter\newcommand \csname f#1\endcsname{\mathfrak{#1}}}
\docsvlist{
a, A, b,B, c, D, e, E, f,F,g,h, j, k, C, l, K, S, N,  Y,  m, M, n,o,p,q,r,s,t,u,v,w, X, 
gl, sl, T, so, su,z,Z, sp}

\renewcommand*{\do}[1]{
\expandafter\newcommand \csname c#1\endcsname{\mathcal{#1}}}
\docsvlist{A,B,C,D,E,F,G,H,I,J,K,L,M,N,O,P,Q,R, RP, DP, S,T,U,V,W,X,Y,Z, GW}

\renewcommand*{\do}[1]{
\expandafter\newcommand \csname
b#1\endcsname{\mathbf{#1}}}
\docsvlist{A,B,C,D,E,F,G,H,I,J,K,L,M,
N,O,P,Q,R, RP, DP, S,T,U,V,W,X,Y,Z, GW}

\newcommand{\DeclareMyOperator}[1]{
\expandafter\DeclareMathOperator\csname #1\endcsname{#1}}
\newcommand{\DeclareMathOperators}{\forcsvlist{\DeclareMyOperator}}

\DeclareMathOperators{Isom, Aut, ch, KU, map,  End, Sym, Hom, Map, Res, Rep, Spec, CAT, Spa, Spd, Gr, PGL, MV, act, Shv, ad, coker, Frob, id, pf, cts, cHk, Fct, cWhit, res, Fr, Fun, dR, Norm, GL, pr, Ind, Stab, pt, tr}

\newcommand{\Sch}{\mathsf{Sch}}
\newcommand{\Aff}{\mathsf{Aff}}
\newcommand{\Set}{\mathsf{Set}}

\renewcommand*{\do}[1]{
\expandafter\newcommand \csname #1\endcsname{\mathrm{#1}}}
\docsvlist{tot, is-contr}
\DeclareMathOperator{\et}{\acute{e}t}

\newcommand{\set}[1]{\left\{#1\right\}}

\newcommand{\wt}{\widetilde}

\newcommand{\bs}{\backslash}

\newcommand{\dra}{\dashrightarrow}
\newcommand{\ra}{\rightarrow}

\newcommand{\xra}[1]{\xrightarrow{#1}}

\newcommand{\thra}{\twoheadrightarrow}
\newcommand{\hra}{\hookrightarrow}

\newcommand{\tx}[1]{\text{#1}}

\newcommand{\QQellbar}{\overline\QQ_\ell}

\DeclareMathSymbol{\mh}{\mathord}{operators}{`\-}

\newcommand{\anbr}[1]{\langle #1 \rangle}
\newcommand{\crbr}[1]{\left \{ #1 \right \} }
\newcommand{\smbr}[1]{\left( #1 \right) }

\usepackage{etoolbox}
\makeatletter
\patchcmd\@thm
  {\let\thm@indent\indent}{\let\thm@indent\noindent}
  {}{}
\makeatother

\title{Geometric Casselman--Shalika in mixed characteristic}

\author{Ashwin Iyengar}
\address{American Mathematical Society}
\email{ashwin.iyengar1@gmail.com}

\author{Milton Lin}
\address{Johns Hopkins University}
\email{clin130@jhu.edu}

\author{Konrad Zou}
\address{Universität Bonn}
\email{kzou@math.uni-bonn.de}

\begin{document}

\begin{abstract}
We establish a geometric analog of the Casselman--Shalika formula for a split connected reductive group over a mixed characteristic local field. In particular, we construct sheaves on the Witt vector affine Grassmannian which geometrize the Fourier coefficients of spherical Hecke operators, and compute their cohomology.
\end{abstract}

\maketitle

\tableofcontents

\section{Introduction}

The goal of this article is to prove a geometric version of the Casselman--Shalika formula for a split connected reductive group $G$ over a mixed characteristic local field.

The original Casselman--Shalika formula (cf. \cite{CS}) gives an explicit combinatorial formula for the values of unramified Whittaker functions, particularly those appearing in local components of generic automorphic representations, in order to compute their local $L$-factors. This formula is equivalent to an explicit description of the action of the spherical Hecke algebra for $G$ on the space of unramified Whittaker functions, which is the version that we consider in this paper.

The ``geometrization'' of this formula occurs over the Witt vector affine Grassmannian for $G$, first defined by Zhu in \cite{Zhu14} and further studied in \cite{BS15}. It is analogous to the geometrization carried out in \cite{NP00} for equi-characteristic local fields over the usual affine Grassmannian, which is a shadow of the action of the spherical Hecke category on the \textit{Whittaker category} constructed in \cite{FGV99}. This has played an important role in the geometric Langlands program of Beilinson--Drinfeld, see \cite{ArinGaits}, \cite{faergeman2022nonvanishing}, and also the local geometric Langlands program, \cite{Beraldo2017}, \cite{Campbell2021}.  Our geometrization of this formula in mixed characteristic suggests the existence of a mixed-characteristic Whittaker category with the action of the spherical Hecke category.

In a sequel to this paper, we will use our main theorem to prove an equivalence between the spherical Hecke category and a mixed characteristic version of the Iwahori--Whittaker category, following \cite{BGMRR}.

In the remainder of the introduction we give an overview of our geometrization, and explain how it recovers the original formula.

\subsection{Main results}

Fix two distinct primes $p \neq \ell$, a split connected reductive group $G$ over\footnote{This is for simplicity of exposition in the introduction. In the paper we work over arbitrary $p$-adic fields.} $\QQ_p$, a split maximal torus $T$ and a Borel $B$ containing $T$ with unipotent radical $N$. The \textit{Witt vector affine Grassmannian} is an ind-(perfect scheme) $\Gr_G$ over $\FF_p$ whose rational points are identified as
    \[ \Gr_G(\FF_p) = G(\QQ_p)/G(\ZZ_p). \]
By the Cartan decomposition $\Gr_G$ admits an stratification into \textit{affine Schubert cells} indexed by dominant cocharacters:
\[
    \Gr_G = \bigsqcup_{\lambda \in X_*(T)_+} \Gr_\lambda, \quad \quad \Gr_\lambda(\FF_p) = G(\ZZ_p) p^\lambda G(\ZZ_p)/G(\ZZ_p)
\]
Denote by $\Gr_{\le \lambda} = \overline\Gr_\lambda$ the closure of each cell, which is a union of lower dimensional strata. Let $\cA_\lambda$ denote its ($\ell$-adic) intersection cohomology sheaf. By the Iwasawa decomposition, we have a stratification into \textit{semi-infinite orbits} indexed by all cocharacters:
\[
    \Gr_G = \bigsqcup_{\nu \in X_*(T)} S_\nu, \quad \quad S_\nu(\FF_p) = N(\QQ_p)p^\nu G(\ZZ_p)/G(\ZZ_p).
\]
The sheaves that we study are supported on their intersection, which we denote by
    \[ \MV_{\lambda,\nu} := \Gr_{\le \lambda} \cap S_\nu. \]
In \autoref{character_sheaf}, associated to each $\Gr_{\le \lambda} \cap S_\nu$, we will construct a map
\[
    h_0^{\lambda,\nu}: \MV_{\lambda,\nu} \to \Gr_{\GG_a}
\]
which sends $n p^\nu$ to $h(n)$, where $h$ is the map
    \[ h: LN \to LN/[LN,LN] \xra\sim \bigoplus_{\alpha \in \Delta} L\GG_a \xra+ L\GG_a \thra L\GG_a/L^+\GG_a \to \Gr_{\GG_a}. \]
Here $L(\cdot)$ and $L^+(\cdot)$ denotes the loop space and arc space respectively, and $\Delta = \Delta(B)$ the set of simple roots for $B$. We then define a character sheaf (i.e. a multiplicative rank 1 local system) $\cL_\psi$ on\footnote{As we explain in \autoref{character_sheaf}, for technical reasons this is really a sheaf on a finite dimensional piece of $\Gr_{\GG_a}$.} $\Gr_{\GG_a}$, which geometrizes a fixed nontrivial character $\psi: \QQ_p \to \QQellbar^\times$.

Now we can state our main theorem.

\begin{theorem}[\autoref{p:dominant_equal_case_main_theorem}, \autoref{p:nu_le_lambda}]
\label{t:main_theorem}
If $\lambda,\nu \in X_*(T)$ are two dominant cocharacters, then
\[
    H^i_c(\MV_{\lambda,\nu} \times_{\Spec \FF_p} \Spec \overline\FF_p, \cA_\lambda \otimes (h_0^{\lambda,\nu})^*\cL_\psi) =
    \begin{cases}
        \QQellbar(-\anbr{\rho,\nu})
        & \text{$i = \anbr{2\rho,\nu}$ and $\nu = \lambda$} \\
        0 & \text{otherwise}
    \end{cases}
\]
where $\rho$ is the half-sum of positive roots and $\smbr{-\anbr{\rho,\nu}}$ denotes a Tate twist.
\end{theorem}

As a corollary of this, we give a new proof of \cite[Theorem 5.2]{FGKV} (\autoref{thm:classical_CS} in our paper) in mixed characteristic, which is equivalent to the classical Casselman--Shalika formula. Recall that the \textit{spherical Hecke algebra} for $G$ is
\[
    \cHk := \Fct_c(G(\ZZ_p)\bs G(\QQ_p)/G(\ZZ_p), \QQellbar)
\]
with its convolution structure, and the space of \textit{unramified (compactly-supported) Whittaker functions} is
\[
    \cWhit := \Fct_c((N(\QQ_p),\psi) \bs G(\QQ_p)/G(\ZZ_p), \QQellbar).
\]
The algebra $\cHk$ admits a basis $\set{H_\lambda}_{\lambda \in X_*(T)_+}$ coming from the Satake isomorphism. The space $\cWhit$ admits a basis $\set{\phi_\nu}$ such that $\phi_\nu$ is supported on the double coset $N(\QQ_p) p^\nu G(\ZZ_p)$ and is uniquely determined by setting
    \[ \phi_\nu(p^\nu) = q^{-\anbr{\rho,\nu}}. \]
There is a natural (right) convolution action $\star$ of $\cHk$ on $\cWhit$.

\begin{theorem}[\autoref{thm:hecke_action_whittaker_classical}]
\label{thm:classical_CS}
For all $\lambda \in X_*(T)_+$,
\[
    \phi_0 \star H_\lambda = \phi_\lambda.
\]
\end{theorem}
This follows from applying the sheaf-function dictionary to $\cA_\lambda \otimes (h_0^{\lambda,\nu})^*\cL_\psi$, see \autoref{trace_function_dictionary}.

\subsection{Comparison with Ng\^{o}--Polo}

In \cite{NP00} the geometric Casselman--Shalika formula is proven for split groups over $\FF_q((t))$. Our proof of \autoref{t:main_theorem} is similar in spirit to \textit{op. cit}, and involves an elaborate series of reduction steps to simpler specific cases. However, there are two significant differences between the proofs, which we explain here.

\subsubsection{There is no ``residue map''.}
In both equal and mixed characteristic, one defines the map
\[
    h: LN \to LN/[LN,LN] \xra\sim \sum_{\alpha \in \Delta} L\GG_a \xra{+} L\GG_a.
\]
In equal characteristic, one can precompose this map with the \textit{residue map}
\[
    L\GG_a \to \GG_a, \quad\quad \sum_{n \in \ZZ} x_nt^n \mapsto x_{-1}
\]
which gives a map $LN \to L\GG_a \xra{\res} \GG_a$. Using the Artin--Schreier cover (i.e. the Lang isogeny $x \mapsto \Fr_q(x)x^{-1}$) a character $\psi: \FF_p \to \QQellbar^\times$ gives rise to a rank 1 local system $\cL_\psi$ on $\GG_a$, which we can then pull back.

In mixed characteristic, this fails for two reasons. The first is that there is no residue map; in fact, there is no nontrivial group homomorphism $\QQ_p \to \FF_p$. The second is that a nontrivial group homomorphism $\psi: \QQ_p \to \QQellbar^\times$ cannot factor through \textit{any} finite quotient of $\QQ_p$. However, without loss of generality, $\psi$ factors as
\[
    \QQ_p \to \QQ_p/\ZZ_p \xra{\psi} \QQellbar^\times.
\]
We thus geometrize $\psi$ by considering the map
    \[ LN \xra{h} L\GG_a \to L\GG_a/L^+\GG_a. \]
and constructing a rank 1 local system $\cL_\psi$ on $L\GG_a/L^+\GG_a$. This presents technical difficulties since $L\GG_a/L^+\GG_a$ is an ind-(perfect scheme). However, in practice we we only care about the restriction of $h$ to the finite-dimensional piece $\Gr_{\leq\lambda} \cap S_\nu$, and this restriction factors through a finite-dimensional subscheme $L^{\geq-s}\GG_a/L^+\GG_a$ for some $s > 0$ (whose $\FF_p$-points are $p^{-s}\ZZ_p/\ZZ_p$):
\[
\begin{tikzcd}
    \Gr_{\leq\lambda} \cap S_\nu \rar \ar[dr, dashed] & L\GG_a/L^+\GG_a \\
    &  L^{\geq-s}\GG_a/L^+\GG_a \ar[u, hook]
\end{tikzcd}
\]
This subscheme admits a theory of character sheaves following Lusztig \cite{Lusztig2006}, so we are reduced to geometrizing (see \autoref{character_sheaf})
    \[ \psi|_{p^s\ZZ_p/\ZZ_p}: p^s\ZZ_p/\ZZ_p \to \QQellbar^\times. \]

\subsubsection{There is no Birkhoff decomposition.}

In \autoref{convolution_products}, we show how we can reduce the cohomological computation to (quasi-)minuscule $\lambda$.
If $\nu_\bullet$ and $\lambda_\bullet$ are two sequences of cocharacters in $X_*(T)$ then one can define the twisted product
\[
    \Gr_{\le \lambda_\bullet} \cap S_{\nu_\bullet} := \smbr{\Gr_{\le \lambda_1} \cap S_{\nu_1}} \wt{\times} \cdots \wt{\times} \smbr{\Gr_{\le \lambda_n} \cap S_{\nu_n}}
\]
which is a subspace of the convolution Grassmannian. As pointed out in \cite[Remark 2.6]{Zhu14}, it is unclear whether this twisted product splits, unlike in \cite[Lemma 9.1]{NP00}. Hence, our argument requires constructing a fiber bundle over $\Gr_{\leq\lambda_\bullet} \cap S_{\nu_\bullet}$ which splits the twisted product into a regular product; this is addressed in \autoref{auxiliary_torsors}. This was inspired by the proof of \cite[Corollary 2.17]{Zhu14}.

The essential obstacle to the splitting of the above twisted product is that there is no ``negative loop group'' in mixed characteristic, which means that there is no analogue of the Birkhoff decomposition for $p$-adic groups. This also breaks many of the basic lemmas in \cite{NP00}, and requires us to reprove basic facts about $\Gr_G$ and its subspaces, many of which are already used in \cite{Zhu14}.

\subsection{Strategy of the argument}

First recall the following generalization of geometric Casselman--Shalika, which we state as a conjecture.

\begin{conjecture}
\label{conj:full_CS}
If $\lambda \in X_*(T)$ is a dominant coweight and $\nu, \mu \in X_*(T)$ are two coweights such that $\nu+\mu$ is dominant, then there is a canonical isomorphism
\[
    H^i_c(\MV_{\lambda,\nu}, \cA_\lambda \otimes (h_\mu^{\lambda,\nu})^*\cL_\psi) \xra\sim
    \begin{cases}
        \Hom_{\widehat{G}}(V^\lambda \otimes V^\mu, V^{\mu+\nu})(-\anbr{\rho,\nu})
   &  i = \anbr{2\rho,\nu} \text{ and } \mu \in X_*(T)_+ \\
        0 & \text{otherwise}
    \end{cases}
\]
where $V^\alpha$ denotes the algebraic representation of $\hat{G}$ of highest weight $\alpha$.
\end{conjecture}

The analogous statement in equal characteristic is proven in \cite{FGV99}.

The map $h_\mu^{\lambda,\nu}$ is the composition of $h_0^{\lambda,\nu}$ and the adjoint action of $p^\mu$. Setting $\mu = 0$ recovers the statement of \autoref{t:main_theorem}. Applying the sheaf-function dictionary to \autoref{conj:full_CS} yields the equality
\[
    \phi_\nu \star H_\lambda = \sum_{\mu \in X_*(T)_+} \dim \Hom_{\hat G}(V^\lambda \otimes V^\nu, V^\mu) \phi_\mu,
\]
which we note can also be immediately derived from \autoref{thm:classical_CS} and an understanding of how tensor products in $\Rep(\hat G)$ decompose into irreducibles. Although we don't prove the full conjecture, the proof of \autoref{t:main_theorem} boils down to proving certain special cases of \autoref{conj:full_CS}, which we outline below.

\begin{itemize}
    \item Suppose $\mu$ is not dominant. Then a twisted equivariance argument shows that the cohomology vanishes in all cases. This is explained in \autoref{sec:non_dominant_case}.

    \item Suppose $\mu$ is dominant and $\lambda$ is minuscule. Then $\Gr_{\leq\lambda} \cap S_\nu$ is nonempty if and only if $\nu \in W \cdot \lambda$, where $W$ denotes the Weyl group of $G$. This case is handled in \autoref{dominant_case_equal_cocharacters} and \autoref{dominant_case_weyl_orbit}. In this case
        \[ \Gr_{\leq\lambda} \cap S_{w\lambda} = \Gr_{\lambda} \cap S_{w\lambda} \]
    which simplifies its geometry considerably --- in this case, we show that $(h_\mu^{\lambda,\nu})^*\cL_\psi$ is the constant sheaf, which reduces us to computing cohomology of $\cA_\lambda$, which is done in \cite{Zhu14}.
    \item Suppose $\mu$ is dominant and $\lambda$ is quasi-minuscule. Then $\Gr_{\leq\lambda} \cap S_\nu$ is nonempty if and only if $\nu \in W \cdot \lambda \cup \set{0}$.
    \begin{itemize}
        \item Suppose $\nu = w\lambda$. Then this is handled in the same way as the minuscule case above.
        \item Suppose $\nu = 0$. Then $\Gr_{\leq\lambda} \cap S_0$ is much more complicated and contains a non-smooth point $\Gr_0$, so requires a more careful analysis performed in \autoref{sec:zero_orbit}. Additionally, we only treat the case where $\mu = 0$. We use a resolution of singularities $\wt\Gr_{\leq\lambda} \to \Gr_{\leq\lambda}$ constructed in \cite{Zhu14} using parahoric subgroups of $G(\ZZ_p)$, and explicitly compute the restriction of $h_0^{\lambda,\nu}$ to various pieces of the pullback of $\Gr_{\leq\lambda} \cap S_\nu$ to $\wt\Gr_\lambda$.
    \end{itemize}
    \item Finally, suppose $\mu = 0$ and $\lambda$ is an arbitrary cocharacter. We first reduce to the case where $\lambda$ is either minuscule or quasi-minuscule. This reduction is carried out in \autoref{sec:break_down_of_cohomology}, and involves finding $\cA_\lambda$ in the convolution of a collection $\cA_{\lambda_i}$ with $\lambda_i$ (quasi-)minuscule, and then transporting sheaves across various auxiliary torsors involved in the construction of convolution products of $\Gr_{\leq\lambda_i} \cap S_{\nu_i}$. These results rely on explicit descriptions of $\Gr_{\leq\lambda_i} \cap S_{\nu_i}$ in the (quasi-)minuscule case which we explain in detail in \autoref{structure_of_MV_Lambda_wLambda}.
\end{itemize}

In \autoref{sec:Notation}, we review notation.
In \autoref{sec:Witt_vector_affine_Grassmannian}, we introduce the Witt vector affine Grassmannian and define the character sheaf $h$ and its twisted version, $h_\mu$. Finally, we apply the results to recover the classical Casselman-Shalika formula in \autoref{trace_function_dictionary}.

\subsection*{Acknowledgements}

We thank Rahul Dalal, David Gepner, Pol van Hoften, James Newton, Carsten Peterson, Toan Q. Pham, Sam Raskin, Yiannis Sakellaridis, Loren Spice, and Xinwen Zhu for helpful conversations and correspondence. We also thank Ceremony Coffee for good working conditions, the Strawberry Sublime, and the 12oz regular latte.

The third named author acknowledges funding from the DFG Leibniz Preis through Peter Scholze.

\newpage

\section{Notation}
\label{sec:Notation}

Fix a prime number $p$ and let $q$ denote a positive power of $p$. Let $k = \overline\FF_q$. Let $W(-)$ denote the $p$-typical Witt vectors. Fix a finite totally ramified extension $F_0/W(\FF_q)[1/p]$ and let $\cO_0 \subset F_0$ denote its ring of integers. Then let $F$ denote the composite of $F_0$ and $W(k)$, and let $\cO$ denote its ring of integers.

Fix a prime number $\ell \neq p$.

We will freely use the $\ell$-adic sheaf theory developed in \cite[Appendix A]{Zhu14} for perfect $k$-schemes. We briefly recall the set up:
\begin{itemize}
    \item Let $\Sch^{\pf}_k$ denote the category of perfect schemes and $D^b_c(X,\QQellbar)$ the bounded derived category of $\ell$-adic sheaves on $X \in \Sch^{\pf}_k$.
    \item Fix a square root of $q$ to define a half Tate twist, $\QQellbar (\frac{1}{2})$.
\end{itemize}

We also fix some group theory notation.
\begin{itemize}
    \item Let $G$ be a split reductive group scheme over $\Spec \cO$, and let $\hat G$ denote its dual group, defined over $\QQellbar$. Fix a split maximal torus $T \subset G$ and a Borel $B \subset G$ containing it. Let $N$ denote the unipotent radical of $B$.
    \item Let $B^-$ denote the Borel opposite to $B$, and let $N^-$ denote its unipotent radical.
    \item Let $\Phi$ denote the set of all roots in the root system for $G$, and let $\Phi_+$ denote the set of positive roots corresponding to $B$. Let $W$ denote the corresponding Weyl group. Let $\Delta \subset \Phi_+$ denote the set of simple roots.
    \item For each $w \in W$ we pick a lift in $G(\cO)$, which we abusively also denote by $w \in G(\cO)$.
    \item Let $n_\alpha: N_\alpha \hra G$ denote the inclusion of the root subgroup corresponding to $\alpha \in \Phi$.
    \item Let $\rho = \frac{1}{2}\sum_{\alpha \in \Phi_+} \alpha$.
    \item Let $\Phi^{\vee}_+$, denote the set of positive coroots. This determines an order on $X_\bullet$: $\lambda \ge \lambda '$ if and only if $\lambda - \lambda ' \in \NN  \Phi^\vee_+$.
    \item Let $X_*(T) = \Hom(\GG_m, T)$ denote the lattice of cocharacters, and let $X_*(T)_+$ denote the cone of dominant cocharacters corresponding to $B$.
    \item Write $\leq$ for the usual Bruhat order with respect to the positive roots.
    \item If $\lambda \in X_*(T)_+$, let $V^\lambda$ denote the irreducible representation of $\hat G$ of highest weight $\lambda$. The set of weights appearing in $V^\lambda$ is
        \[ \Omega(\lambda) = \set{\mu \in X_*(T) : w\mu \leq \lambda \text{ for all } w \in W}. \]
    \item If $\nu \in X_*(T)$, we write $\varpi^\nu := \nu(\varpi)$.
    \item In general, if $H$ is a group and $h \in H$ we let $\ad(h)$ denote conjugation by $h$.
    \item For $\lambda \in X_*(T)$, let $P_\lambda$ denote the parabolic subgroup of $G$ generated by the root subgroups $N_\alpha$ of $G$ corresponding to those roots $\alpha$ satisfying $\anbr{\alpha, \lambda} \leq 0$.
\end{itemize}

\begin{definition}
\label{def:minuscule_quasiminuscule}
Let $M$ denote the set of minimal elements for the ordering on the set $X_*(T)_+ \setminus \set{0}$. If $\lambda \in M$, then either
\begin{enumerate}
    \item $\Omega(\lambda) = W \cdot \lambda$, in which case we say that $\lambda$ is \textit{minuscule}, or
    \item $\Omega(\lambda) = W \cdot \lambda \cup \set{0}$, in which case we say that $\lambda$ is \textit{quasi-minuscule}.
\end{enumerate}
\end{definition}

There is a useful alternative characterization, as follows.
\begin{p}
\label{alternative_characterization_quasi_minuscule}
    Let $\mu \in M$.
    \begin{enumerate}
        \item $\lambda$ is minuscule if $\anbr{\alpha,\lambda}\in \crbr{0,\pm 1}$ for all $\alpha \in \Phi$.
        \item $\lambda$ is quasi-minuscule if there exists a unique root $\gamma$ such that $\anbr{\gamma,\lambda} \ge 2$. In this case, $\gamma = \lambda^\vee$.
    \end{enumerate}
\end{p}

\newpage

\section{Witt vector affine Grassmannian}
\label{sec:Witt_vector_affine_Grassmannian}

In this section we introduce the Witt vector affine Grassmannian and study the relevant subspaces.

\begin{definition}
If $R$ is a perfect $k$-algebra, write
    \[ W_\cO(R) = W(R) \otimes_{W(k)} \cO \]
We also define the truncated Witt vectors
    \[ W_{\cO,h}(R) = W(R) \otimes_{W(k)} \cO/\varpi^h. \]
\end{definition}

\begin{definition}[{{\cite[Section 1]{Zhu14}}}]
\label{def:standard_constructions}
$ $
\begin{itemize}
    \item Let $\Aff_F^{\tx{ft}}$ and $\Aff_\cO^{\tx{ft}}$ denote the categories of finite type schemes over $F$ and $\cO$, respectively. Let $\Aff_k^{\pf}$ denote the category of perfect affine $k$-schemes.
    \begin{align*}
        L^h: \Aff^{\tx{ft}}_\cO &\to \Fun(\Aff^{\pf}_k, \Set) \\
        \cX &\mapsto L^h\cX(R) = X(W_{\cO,h}(R)) \\
        L^+:\Aff^{\tx{ft}}_\cO &\to \Fun(\Aff^{\pf}_k, \Set) \\
        \cX &\mapsto L^+\cX(R):= \cX(W_{\cO}(R)) = \varprojlim_h L^h\cX(R) \\
        L: \Aff^{\tx{ft}}_F &\to \Fun(\Aff^{\pf}_k, \Set) \\
        X &\mapsto LX(R) = X(W_\cO(R)[1/\varpi])
    \end{align*}
    The space $L^h\cX$ is called the ($h$-)truncated positive loop space, and is a perfect $k$-scheme. The space $L^+\cX$ is called the positive loop space, and is a perfect $k$-scheme. The space $LX$ is called the loop space, and is an ind-(perfect $k$-scheme). As a consequence of \cite{greenberg1961schemata}, we have
    \[
        L^+\cX \simeq \varprojlim_{h} L^h\cX.
    \]
    If $X$ or $\cX$ is a group, then the resulting loop objects are groups as well.
    \item If $H$ is any smooth affine group scheme over $\cO$, define the \'etale quotient
        \[ \Gr_H = LH/L^+H \]
    called the \textit{Witt vector affine Grassmannian for $H$}.
\end{itemize}
\end{definition}

\begin{definition}
\label{def:finite_type_perfect_scheme}
A perfect $k$-scheme $X$ is
    \begin{itemize}
        \item \textit{locally of finite type} if exists an \'etale cover $\crbr{U_i}_{i \in I}$ of $X$ such that each $U_i$ is the perfection of an affine scheme,
        \item \textit{perfectly of finite type} if it is locally perfect of finite type and quasi-compact,
        \item \textit{perfectly of finite presentation} if it is perfectly of finite type and quasi-separated.
    \end{itemize}
\end{definition}

By the Cartan decomposition, $\Gr_G$ can be written as the colimit of perfections of projective varieties, called \textit{affine Schubert varieties}:
    \[ \Gr_G = \varinjlim_{\lambda \in X_*(T)_+} \Gr_{\leq \lambda} \]
and that the Schubert varieties are the closure of their maximal Schubert cells:
    \[ \Gr_{\leq \lambda} = \overline{\Gr_{\lambda}} = \bigcup_{\lambda' \leq \lambda} \Gr_{\lambda'}. \]
The inclusion $\Gr_\lambda \subset \Gr_G$ is a locally closed embedding. The $k$-points of the Schubert cells are
    \[ \Gr_\lambda(k) = G(\cO)\varpi^\lambda G(\cO), \]
By definition there is a left action of $LG$ on $\Gr_G$. This restricts to an action of $L^+G$ on $\Gr_{\leq \lambda}$.

\begin{lemma}\label{l:schubert_action_factors}
The action of $L^+G$ on $\Gr_{\leq\lambda}$ factors through $L^hG$ for $h$ large enough.
\end{lemma}
\begin{proof}
This is explained in the proof of \cite[Proposition 1.23]{Zhu14}.
\end{proof}

\begin{definition}
\label{def:IC_sheaves}
For $\lambda \in X_*(T)_+$ we let $\cA_\lambda$ denote the intersection cohomology sheaf on $\Gr_{\leq\lambda}$. This is defined as the intermediate extension of $\QQellbar[\anbr{2\rho, \lambda}]((\rho, \lambda))$ on $\Gr_{\lambda}$ to all of $\Gr_{\leq\lambda}$.
\end{definition}
Let $P_{L^+G}(\Gr_G)$ denote the category of $L^+G$-equivariant perverse sheaves on $\Gr_G$, as defined in \cite[Section 2]{Zhu14}. We have
    \[ \cA_\lambda \in P_{L^+G}(\Gr_G), \]
and that the irreducible objects in $P_{L^+G}(\Gr_G)$ are exactly the $\cA_\mu$ for $\mu \in X_*(T)_+$.

The inclusion $N \hra G$ functorially induces an inclusion $\Gr_N \hra \Gr_G$.

\begin{definition}
The \textit{semi-infinite orbit} of a cocharacter $\nu \in X_*(T)$ is
    \[ S_\nu = \varpi^\nu \Gr_N \subset \Gr_G. \]
\end{definition}

\begin{definition}\label{def:MV_cycles}
Let
    \[ \MV_{\lambda,\nu} := \Gr_{\leq\lambda} \cap S_\nu. \]
We write ``MV'' for ``Mirkovi\'c--Vilonen''. In the literature a \textit{Mirkovi\'c--Vilonen cycle} usually refers to an irreducible component of $\MV_{\lambda,\nu}$, but we use MV to denote the whole intersection.
\end{definition}

\begin{lemma}
\label{lem:MV_is_pfp}
For all $\lambda \in X_*(T)_+$ and all $\nu \in X_*(T)$, $\MV_{\lambda,\nu}$ is a pfp perfect $k$-scheme.
\end{lemma}
\begin{proof}
Consider the diagram
\[
\begin{tikzcd}
    \MV_{\lambda,\nu} \ar[r] \ar[d] \ar[dr, phantom , "\lrcorner"] & S_\nu \ar[d] \\
     \Gr_{\le \lambda} \ar[r]  & \Gr_G
\end{tikzcd}
\]
As explained in \cite[Proposition 1.20]{Zhu14} (cf. \cite[p.173]{Beilinson1996QuantizationOH}), $S_\nu \to \Gr_G$ is a locally closed immersion of perfect ind-schemes, and thus $\MV_{\lambda,\nu} \to \Gr_{\leq\lambda}$ is a locally closed immersion of schemes. Now $\Gr_{\leq\lambda}$ is the perfection of a projective algebraic $k$-variety, and is therefore a pfp perfect scheme, so the result follows from the following two observations:
\begin{enumerate}
    \item $\MV_{\lambda,\nu}$ is quasi-compact and quasi-separated. To see this, note that the underlying topological space of $\Gr_{\le \lambda}$ is Noetherian, since $\Gr_{\leq\lambda}$ is the perfection of a finite type $k$-variety and perfection preserves underlying topological spaces. Now any subspace of Noetherian space is quasi-compact and quasi-separated by \cite[\href{https://stacks.math.columbia.edu/tag/0052}{Lemma 0052}]{stacks-project}.
    \item The property of being locally perfectly of finite type is preserved under base change along open and closed immersions, hence for locally closed immersions as well. \qedhere
\end{enumerate}
\end{proof}

\begin{p}[{{\cite[Cor 2.8]{Zhu14}}}]
\label{p:MV_nonempty}
If $\lambda \in X_*(T)_+$ and $\nu \in X_*(T)$ then
\[
    \MV_{\lambda,\nu} \neq \varnothing \Leftrightarrow \nu \in \Omega(\lambda)
\]
and is equidimensional of rank $\anbr{\rho, \nu+\lambda}$.
\end{p}

\subsection{Character sheaf}
\label{character_sheaf}

Fix, once and for all, an additive character
    \[ \psi: F_0 \to F_0/\cO_0 \to \QQellbar^\times \]
such that\footnote{This conductor will simplify the rest of the arguments and does not amount to any loss of generality in \autoref{t:main_theorem}.} $\psi(p^{-1}\cO) \neq 1$.

In order to geometrize the additive character and consider Whittaker sheaves, we first consider the natural map
\begin{equation}
\label{eq:the_h_map}
    h_\mu: LN \xra{\ad \varpi^\mu} LN \to LN/[LN,LN] \xra\sim \prod_{\alpha \in \Phi_+} L\GG_a \xra+ L\GG_a \to L\GG_a/L^+\GG_a.
\end{equation}
In the existing proofs of the geometric Casselman--Shalika formula in equal characteristic (\cite{FGV99}, \cite{NP00}), the character sheaf is induced from a \textit{residue map}, which is defined by taking
\begin{align*}
    \res: L\GG_a &\to \GG_a \\
    \sum_{n \geq -M} a_n t^n &\mapsto a_{-1},
\end{align*}
which on the level of $k$-points is a group homomorphism $k((t)) \to k$. But in mixed characteristic this doesn't make sense, since there is no nontrivial group homomorphism $\QQ_p \to \FF_p$. Moreover, there is no choice of nontrivial $\psi$ which is supported on a finite subset of $\QQ_p$.

This naturally induces a map on the semi-infinite orbit $S_\nu$ as follows:

\begin{lemma}\label{lemma:definition_of_h_map}
If $\mu \in X_\bullet(T)$ is a character such that $\mu + \nu$ is dominant, then $h$ induces a map
    \[ h_\mu^\nu: S_\nu \to L\GG_a/L^+\GG_a. \]
which on $k$-points yields
\begin{align*}
    N(F) \varpi^\nu N(\cO) / N(\cO) &\to F/\cO \\
    n \varpi^\nu N(\cO) &\mapsto h(\ad(\varpi^\mu)(n)).
\end{align*}

\end{lemma}
\begin{proof}
Note that $S_\nu = \varpi^\nu\Gr_N = (\varpi^\nu LN)/L^+N$. But this is the \'etale sheafification of the na\"ive quotient of presheaves. So for $R$ a perfect $k$-algebra we define
\begin{align*}
    (\varpi^\nu LN(R)) / L^+N(R) &\to L\GG_a(R) / L^+\GG_a(R) \\
    \varpi^\nu n \mod L^+N(R) &\mapsto h(\ad(\varpi^{\mu+\nu})(n)).
\end{align*}
To see that this is well-defined, suppose $\varpi^\nu n L^+N(R) = \varpi^\nu m L^+N(R)$. Then $n^{-1}m \in L^+N(R)$, but $\mu+\nu$ is dominant so $\ad(\varpi^{\mu+\nu})(n^{-1}m) \in L^+N(R)$, which maps to $L^+\GG_a(R)$ under the group homomorphism $h$. This is clearly functorial and extends to a morphism of presheaves, which we then sheafify. The statement on $k$-points follows noting that $\varpi^\nu$ normalizes $N(F)$.
\end{proof}

We want to turn the nontrivial additive character
    \[ \psi: F_0 \to F_0/\cO_0 \to \QQellbar^\times \]
into a character sheaf
on
    \[ \Gr_{\GG_a} := L\GG_a/L^+\GG_a \]
(whose $k$ points are exactly $F/\cO$) and pull it back along $h_\mu^\nu$. However, $\Gr_{\GG_a}$ is a group ind-scheme, and a geometric version of $\psi$ on $\Gr_{\GG_a}$ would have to be supported on every point of $\Gr_{\GG_a}$. While there probably exists a formalism that allows for such sheaves, we will avoid this and construct a family of compatible characters supported on finite pieces, which is enough for our purposes.

\begin{definition}
If $H$ is a smooth affine group scheme over $\cO$ and $s \in \ZZ$, we let $L^{\geq s}H$ denote the image of $L^+H$ under the isomorphism
    \[ LH \xra{\cdot \varpi^s} LH. \]
\end{definition}

For $s > 0$ it's clear that the natural embedding $L^+H \to LH$ factors through $L^{\geq -s}H$, so we can form the quotient
    \[ L^{\geq s}H/L^+H, \]
which is isomorphic to $L^sH$. In particular, $\displaystyle \Gr_{\GG_a} \simeq \varinjlim_s L^{\ge -s}\GG_a/ L^+\GG_a$.

\begin{lemma}
\label{lem:action_factoring_through_finite_part}
If $\lambda \in X_*(T)_+$ and and $\nu \in X_*(T)$, there is a factorization
    \[ \begin{tikzcd}
        \MV_{\lambda,\nu} \rar[dashed]{h_\mu^{\lambda,\nu}} \dar[hook] & L\GG_a^{\geq-s}/L^+\GG_a \dar[hook] \\
        S_\nu \rar[swap]{h_\mu^\nu} & L\GG_a/L^+\GG_a
    \end{tikzcd} \]
where $s > 0$ is some large enough positive integer.
\end{lemma}
\begin{proof}
Note $\MV_{\lambda,\nu}$ is a subscheme of $\Gr_{\leq\lambda}$, which is the perfection of a projective variety over $k$, by the results of \cite{WittProjectivity}, and is therefore quasi-compact over $k$. So the morphism to the ind-scheme
    \[ L\GG_a/L^+\GG_a = \varinjlim_s L^{\geq-s}\GG_a/L^+\GG_a \]
must factor through one of the $L\GG_a^{\geq-s}/L^+\GG_a$.
\end{proof}

\begin{lemma}\label{lem:GrGa-Shift}
The quotient $L^{\geq-s}\GG_a/L^+\GG_a$ is represented by a pfp perfect group scheme and its $k$-points are natukirally identified with $\varpi^{-s}\cO/\cO$.
\end{lemma}
\begin{proof}
We exhibit an isomorphism $L^{\geq-s}\GG_a/L^+\GG_a \xra\sim L^s\GG_a$. If $R$ is a perfect $k$-algebra, we can define an isomorphism of group-valued presheaves
\begin{align*}
    L^{\geq-s}\GG_a(R)/L^+\GG_a(R) &\to L^s\GG_a(R) \\
    \sum_{i=-s}^{-1} [r_i]\varpi^i & \mapsto \sum_{i=0}^{s-1} [r_{i-s}] \varpi^i
\end{align*}
and then take the sheafification. We conclude by noting that $L^s\GG_a$ is the perfection of the finite type group scheme $L^s_p\GG_a$ whose $k$-points are $\cO/\varpi^s\cO$.
\end{proof}

\begin{proposition}[{{\cite[Section 5]{Lusztig2006}}}]
\label{cor:lustzig_character_sheaves}
There is a unique rank 1 $\QQellbar$-local system $\cL_{\psi,s}$ on $L^{\geq-s}\GG_a/L^+\GG_a$ such that
\begin{enumerate}
    \item $a^*\cL_{\psi,s} \cong \cL_{\psi,s} \boxtimes \cL_{\psi,s}$, where $a$ is the addition map on $L^{\geq-s}\GG_a/L^+\GG_a$, where $a$ is the additive structure of $L^{\ge -s}\GG_a/ L^+\GG_a$, and
    \item With respect to the relative Frobenius on $L^{\geq -s}\GG_a/L^+\GG_a$ (coming from $\GG_{a,\FF_q}$), the value of the trace of Frobenius at the stalk of $\cL_{\psi,s}$ at $g \in \varpi^{-s}\cO_0/\cO_0$ is $\psi(g)$.
\end{enumerate}
\end{proposition}

Lusztig did not consider perfections of finite type group schemes in his work, but since the \'etale site is insensitive to perfection and $L^{\geq-s}\GG_a/L^+\GG_a$ is a pfp perfect group scheme, the theorem applies to $L^{\geq-s}\GG_a/L^+\GG_a$ without any further work. See \cite[Theorem 2.9]{deshpande2023character} for another account of this.

Lastly, if $t > s$ there is an inclusion
    \[ \iota: L^{\geq-s}\GG_a/L^+\GG_a \hra L^{\geq-t}\GG_a/L^+\GG_a \]
and $\iota^*\cL_{\psi,r} \simeq  \cL_{\psi,r}$. We will thus abusively denote $\cL_\psi$ as the character sheaf on any of the finite pieces $L^{\ge -s}\GG_a/ L^+\GG_a $.

\newpage

\section{Vanishing of cohomology for non-dominant \texorpdfstring{$\mu$}{mu}}
\label{sec:non_dominant_case}
In this section, we verify \autoref{conj:full_CS} when $\mu \in X_*(T)$ is not dominant.

By \autoref{l:schubert_action_factors} the $L^+G$-action on $\Gr_{\leq \lambda}$ factors through $L^hG$ for some large enough $h > 0$. Therefore, the $L^+N$-action on $\MV_{\lambda,\mu}$ factors through $L^hN$ as well. A direct computation shows that the map $h_\mu|_{L^+N}: L^+N \to L\GG_a/L^+\GG_a$ lands in $L^{\geq -s}\GG_a/L^+\GG_a$ for large enough $s$, and that it factors as
    \[ h_\mu|_{L^+N}: L^+N \thra L^hN \to L^{\geq -s}\GG_a/L^+\GG_a \hra L\GG_a/L^+\GG_a \]
for large enough $h$.

\begin{lemma}
\label{lem:morphism_k_points}
Suppose $X$ and $Y$ are pfp perfect $k$-schemes. If $f, g: X \to Y$ are two morphisms such that $f(k), g(k): X(k) \to Y(k)$ are equal, then $f = g$.
\end{lemma}
\begin{proof}
In \cite[Proposition A.15]{Zhu14} Zhu defines models $X'$ and $Y'$ for $X$ and $Y$, respectively, so that $X',Y'$ are weakly normal (so in particular reduced) finitely presented $k$-schemes satisfying $(X')^{\pf} = X$ and $(Y')^{\pf} = Y$. Then, as explained in \cite[Proposition A.17]{Zhu14}, after possibly twisting the $k$-structure on $X'$ by a high enough power of the Frobenius on $k$ (i.e. replacing the structure map $X' \to \Spec k$ with $X' \xra{\sigma^m} X' \to k$ for some large enough $m$) there exists $k$-morphisms $f', g': X' \to Y'$ such that $(f')^{\pf} = f$ and $(g')^{\pf} = g$. Since $k$ is perfect, the projections $X(k) \to X'(k)$ and $Y(k) \to Y'(k)$ are bijective, and it is clear that $f'(k) = f(k) = g(k) = g'(k)$.

Now observe that $f'$ and $g'$ are morphisms of reduced finite type $k$-schemes over an algebraically closed field. Therefore, $f'(k) = g'(k)$ implies $f' = g'$, so $f = g$.
\end{proof}

\begin{proposition}
\label{p:finite_action_mv_cycles}
Choose $s$ such that $h_\mu|_{L^+N}$ and $h_\mu^{\lambda,\nu}$ both factor through $L^{\geq -s}\GG_a/L^+\GG_a \to L\GG_a/L^+\GG_a$. Then the following diagram commutes:
\[ \begin{tikzcd}
    L^+N  \times \MV_{\lambda,\nu} \rar{\act} \ar[d] & \MV_{\lambda,\nu} \dar[equals] \\
    L^hN \times \MV_{\lambda,\nu} \rar{\act} \dar[swap]{h_\mu \times h_\mu^{\lambda,\nu}} & \MV_{\lambda,\nu} \dar{h_\mu^{\lambda,\nu}} \\
    L^{\geq -s}\GG_a/L^+\GG_a \times L^{\geq -s}\GG_a/L^+\GG_a \rar{+} \ar[d]  & L^{\geq -s}\GG_a/L^+\GG_a  \ar[d] \\
    L\GG_a / L^+\GG_a \times L\GG_a / L^+\GG_a \rar{+} & L\GG_a/ L^+\GG_a
\end{tikzcd} \]
\end{proposition}
\begin{proof}
The top and bottom diagrams commute by construction. Each of $L^hN$ and $L^{-s}\GG_a/L^+\GG_a$ and $\MV_{\lambda,\nu}$ is a pfp perfect $k$-scheme, so by \autoref{lem:morphism_k_points} it suffices to check that the middle diagram commutes on the level of $k$-points. This is a straightforward diagram chase.
\end{proof}

\begin{corollary}
\label{c:nondominant_vanishing}
If $\lambda \in X_*(T)$ is dominant and $\mu,\nu \in X_*(T)$ are such that $\mu$ is not dominant but $\mu+\nu$ is dominant, then
    \[ R\Gamma_c(\MV_{\lambda,\nu}, \cA_\lambda \otimes (h_\mu^{\lambda,\nu})^*\cL_\psi) = 0. \]
\end{corollary}
\begin{proof}
By \autoref{p:finite_action_mv_cycles} and the fact that $\cA_\lambda$ is $L^+G$-equivariant,
\begin{align*}
    \act^*(\cA_\lambda \otimes (h_\mu^{\lambda,\nu})^*\cL_\psi) &= \act^*\cA_\lambda \otimes \act^*(h_\mu^{\lambda,\nu})^*\cL_\psi \\
    &= (\QQellbar \boxtimes \cA_\lambda) \otimes (h_\mu \times h_\mu^{\lambda,\nu})^*(+)^*\cL_\psi \\
    &= (\QQellbar \boxtimes \cA_\lambda) \otimes (h_\mu \times h_\mu^{\lambda,\nu})^*(\cL_\psi \boxtimes \cL_\psi) \\
    &= (\QQellbar \boxtimes \cA_\lambda) \otimes (h_\mu^*\cL_\psi \boxtimes (h_\mu^{\lambda,\nu})^*\cL_\psi) \\
    &= h_\mu^*\cL_\psi \boxtimes (\cA_\lambda \otimes (h_\mu^{\lambda,\nu})^*\cL_\psi),
\end{align*}
so $\cA_\lambda \otimes (h_\mu^{\lambda,\nu})^*\cL_\psi$ is $(L^hN, h_\mu^*\cL_\psi)$-equivariant.

If $\mu$ is not dominant, pick a simple root $\alpha$ such that $\anbr{\alpha,\mu} < 0$. The composition
    \[ L^+\GG_a \hra L\GG_a \xra{Lu_\alpha} LN \xra{\ad\varpi^\mu} LN \ra LN/[LN,LN] \xra{+} L\GG_a \]
is precisely multiplication by $\varpi^{\anbr{\alpha,\mu}}$, so $h_\mu|_{L^+N}$ is non-trivial. This implies that $h_\mu^*\cL_\psi$ is also nontrivial. To see why, note that the local system $h_\mu^*\cL_\psi$ is the character sheaf for the character
    \[ \pi_1^{\et}(L^hN) \to \pi_1^{\et}(L^{\geq-s}\GG_a/L^+\GG_a) \thra \varpi^{-s}\cO/\cO \to \QQellbar^\times \]
and the first map is surjective since the morphism $L^hN \to L^{\geq-s}\GG_a/L^+\GG_a$ has connected geometric fibers, so this character is nontrivial. We conclude by applying \autoref{p:vanishing_cohomology_equivariant_sheaf}.
\end{proof}

\begin{p}\label{p:vanishing_cohomology_equivariant_sheaf}
Suppose $Z$ is a pfp perfect group scheme over $k$ with an action
    \[ \act: G \times Z \ra Z \]
of a pfp perfect group scheme $G$ defined over $k$. If $\cL$ is a non-trivial rank 1 local system on $G$ and $\cF$ is a $(G, \cL)$-equivariant complex of sheaves on $Z$, i.e.
    \[ \act^*\cF \simeq \cL \boxtimes \cF \]
then
    \[ R\Gamma_c(Z, \cF) = 0. \]
\end{p}
\begin{proof}
The proof of \cite[Lemma 3.3]{Ngo98} goes through verbatim, for $G$ a connected commutative algebraic group replacing $\GG_a$, noting that the statement depends only on the \'etale topology, which is insensitive to perfection.
\end{proof}

\newpage

\section{Outline of main argument}
\label{sec:main_argument}

As explained in the introduction, we split cases of $\nu \in \Omega(\lambda)$ according to whether $\nu^+ = \lambda$ or $\nu^+ < \lambda$, where $\nu^+$ is the unique dominant representative of $\nu$ in $W \cdot \nu$. We address the first case in \autoref{dominant_case_equal_cocharacters}, \autoref{dominant_case_weyl_orbit}, and the second case in \autoref{sec:unequal_cocharacters}.

\subsection{Structure of \texorpdfstring{$\MV_{\lambda,w\lambda}$}{MV lambda w lambda}}
\label{structure_of_MV_Lambda_wLambda}
\begin{lemma}
\label{lem:weyl_orbit_intersection}
For all $\lambda \in X_*(T)_+$ and all $w \in W$,
    \[ \MV_{\lambda,w\lambda} = S_{w\lambda} \cap \Gr_\lambda = L^+N\varpi^{w\lambda}L^+G/L^+G = L^+N w \varpi^\lambda L^+G/L^+G. \]
\end{lemma}
\begin{proof}
By \autoref{p:MV_nonempty},
    \[ S_{w\lambda} \cap \Gr_{\le \lambda'} \not = \varnothing \text{ if and only if } w\lambda \in \Omega(\lambda'). \]
If $\lambda' < \lambda$ we cannot have $w\lambda \in \Omega(\lambda')$. So since
\[ \MV_{\lambda,w\lambda} = \bigcup_{\lambda' \leq \lambda} S_{w\lambda} \cap \Gr_{\lambda'} \]  we see that $\MV_{\lambda,w\lambda} = S_{w\lambda} \cap \Gr_\lambda$.

For the second equality note that all of the schemes in question are pfp perfect schemes, so by \autoref{lem:morphism_k_points} it suffices to check the equality on $k$-points. This boils down to an old result of Satake: by the Remark before Section 8.3 in \cite{Satake}\footnote{Satake's paper assumes $F$ is a finite extension of $\QQ_p$, but the same proof works when $F$ is finite and totally ramified over $F_0 = W(k)$, where $k$ is an algebraically closed field of characteristic $p$.},
\begin{equation}
\label{eq:satake_computation}
    N(F)\varpi^{\lambda}G(\cO) \cap G(\cO)\varpi^{\lambda}G(\cO) = N(\cO)\varpi^{\lambda}G(\cO).
\end{equation}
For nontrivial $w \in W$, note that
    \[ N(F) \varpi^{w\lambda} G(\cO) = w N_w(F) \varpi^{\lambda} G(\cO), \]
where $N_w = wNw^{-1}$, which is the unipotent radical of the Borel subgroup $wBw^{-1}$. We thus apply \autoref{eq:satake_computation} with $N_w$ replacing $N$ to obtain
\begin{align*}
    N(F) \varpi^{w\lambda} G(\cO) \cap G(\cO) \varpi^\lambda G(\cO) &= w(N_w \varpi^{\lambda} G(\cO) \cap G(\cO) \varpi^\lambda G(\cO)) \\
    &= w(N_w(\cO) \varpi^\lambda G(\cO)) \\
    &= N(\cO) \varpi^{w\lambda} G(\cO). \qedhere
\end{align*}
\end{proof}

Recall from \cite{Zhu14} that there is an isomorphism
\[
    \Gr_\lambda \xra\sim L^+G/(L^+G \cap \ad(\varpi^\lambda)L^+G) \quad\quad g\varpi^{\lambda} \mapsto g
\]
and a reduction map
    \[ \mathring\phi: \Gr_\lambda \xra\sim L^+G/(L^+G \cap \ad(\varpi^\lambda)L^+G) \to (\bar G/ \bar P_\lambda)^{\pf}. \]
Since by \autoref{lem:weyl_orbit_intersection} we know that $\MV_{\lambda,w\lambda} = L^+N w \varpi^\lambda L^+G/L^+G$, it follows that the reduction map factors as
\[
\begin{tikzcd}
    \MV_{\lambda,w\lambda} \rar[hook]  \ar[dr, phantom, "\lrcorner" ] \dar[swap]{\mathring\phi|_{\MV_{\lambda,w\lambda}}}  & \Gr_\lambda \dar{\mathring\phi} \\
    (\bar N w \bar P_\lambda/\bar P_\lambda)^{\pf} \rar[hook] & (\bar G / \bar P_\lambda)^{\pf}
\end{tikzcd}
\]

Moreover, it is clear from the definition that $\MV_{\lambda,w\lambda} \to (\bar N w \bar P_\lambda/\bar P_\lambda)^{\pf}$ is surjective.

The orbit-stabilizer theorem gives us that
    \[ L^+N/(L^+N \cap \ad(\varpi^{w\lambda})L^+N) \xra\sim \MV_{\lambda,w\lambda} \]
so by factoring into positive root subgroups we find that (see \cite[Lem. 7.4]{Ngo98} for an analogous statement in equal characteristic).
\begin{align*}
    L^+N/(L^+N \cap \ad(\varpi^{w\lambda})L^+N) &\simeq \prod_{\alpha \in \Phi^+} L^+N_\alpha/\varpi^{\max(\anbr{w^{-1}\alpha,\lambda},0)}L^+N_\alpha \\
    & = \prod_{\substack{\alpha \in \Phi^+, \\ \anbr{w^{-1}\alpha,\lambda} > 0}} L^+N_\alpha/\varpi^{\anbr{w^{-1}\alpha,\lambda}} L^+N_\alpha
\end{align*}
Similarly,
    \[ \bar N/(\bar N \cap \ad(w)\bar P_\lambda) \xra\sim \bar N w \bar P_\lambda / \bar P_\lambda\]
and again factoring into positive root subgroups we find that
\[
    \bar N / (\bar N \cap \ad(w) \bar P_\lambda) \simeq \prod_{\substack{\alpha \in \Phi^+, \\ \anbr{w^{-1}\alpha,\lambda} > 0}} \bar N_\alpha
\]

\begin{lemma}
\label{lem:mv_weyl_reduction_commutes}
The following diagram commutes:
\[
\begin{tikzcd}
    \prod_{\substack{\alpha \in \Phi^+, \\ \anbr{w^{-1}\alpha,\lambda} > 0}} L^+N_\alpha/\varpi^{\anbr{w^{-1}\alpha,\lambda}} L^+N_\alpha \rar{\sim} \dar & \MV_{\lambda,w\lambda} \dar{\mathring\phi} \dar \\
    \prod_{\substack{\alpha \in \Phi^+, \\ \anbr{w^{-1}\alpha,\lambda} > 0}} \bar N_\alpha^{\pf} \rar{\sim} &  (\bar N w \bar P_\lambda/\bar P_\lambda)^{\pf}
\end{tikzcd}
\]
where the left vertical arrow is reduction mod $\varpi$ in each factor.
\end{lemma}
\begin{proof}
Omitted.
\end{proof}

From this, we can draw various conclusions about the behavior of the reduction map.

\begin{lemma}
If $\lambda$ is minuscule, then the reduction map
    \[ \MV_{\lambda,w\lambda} \to (\bar N w \bar P_\lambda/\bar P_\lambda)^{\pf} \]
is an isomorphism.
\end{lemma}
\begin{proof}
Since $\lambda$ is minuscule, it pairs to $1$ with any positive root, so apply \autoref{lem:mv_weyl_reduction_commutes}.
\end{proof}

\begin{lemma}
\label{lem:quasi_minuscule_semiinfinite_orbit_bundle}
If $\lambda$ is quasi-minuscule, then the reduction map
    \[ \MV_{\lambda,w\lambda} \to (\bar N w \bar P_\lambda/P_\lambda)^{\pf} \]
is:
\begin{enumerate}
    \item an isomorphism if $w\lambda^\vee \in \Phi_-$, and
    \item isomorphic to the projection
    \[
        (\bar N w \bar P_\lambda / \bar P_\lambda)^{\pf} \times \GG_a^{\pf} \to (\bar N w \bar P_\lambda / \bar P_\lambda)^{\pf}
    \]
    if $w\lambda^\vee \in \Phi_+$.
\end{enumerate}
\end{lemma}
\begin{proof}
Since $\lambda$ is quasi-minuscule, we know that $\anbr{\alpha,\lambda} \in \set{0,1}$ for all $\alpha \in \Phi^+ \setminus \set{\lambda^\vee}$ and $\anbr{\lambda^\vee,\lambda} = 2$. So we conclude by noting that $N_\alpha \simeq \GG_a$ and applying \autoref{lem:mv_weyl_reduction_commutes}.
\end{proof}

Note that if $\lambda$ is quasi-minuscule
\[
    \mathring\phi: \Gr_\lambda \to (\bar G / \bar P_\lambda)^{\pf}
\]
is an $(\AA^1)^{\pf}$-bundle. Even though $\MV_{\lambda,w\lambda}$ is only a section of this bundle restricted to $(\bar N w \bar P_\lambda/\bar P_\lambda)^{\pf}$ when $w\lambda^\vee \in \Phi_-$, we show that the whole restricted bundle is trivial. This will be used later on in \autoref{sec:zero_orbit} when we analyze $\MV_{\lambda,0}$ for $\lambda$ quasi-minuscule.

Before showing this, we construct an appropriate parahoric subgroup. If $\lambda \in X_*(T)_+$, recall that we had a parabolic subgroup $P_\lambda$ generated by $T$ and the root subgroups $N_\alpha$ for $\alpha$ satisfying $\anbr{\alpha,\lambda} \leq 0$.

\begin{definition}
\label{def:p_lambda}
Define the parahoric subgroup
\[
    \cP_\lambda(\cO) = \left\langle T(\cO), N_\alpha(\varpi^a \cO) \text{ where } a = \begin{cases} 0 & \anbr{\alpha,\lambda} \leq 0 \\ 1 & \anbr{\alpha,\lambda} > 0 \end{cases}\right\rangle \subset G(\cO).
\]
This uniquely determines a parahoric $\cO$-group scheme $\cP_\lambda$.
\end{definition}

Reduction mod $\varpi$ gives a surjective map $\cP_\lambda(\cO) \to P_\lambda(k)$\footnote{In fact, $\cP_\lambda(\cO)$ is equal to the preimage of $P_\lambda(k)$ under the reduction $G(\cO) \to G(k)$.}.

Let $\cL_w := \mathring\phi^{-1}((\bar N w \bar P_\lambda/\bar P_\lambda)^{\pf})$.

\begin{lemma}
\label{identification_of_Lw}
\begin{enumerate}
    \item If $w\lambda^\vee \in \Phi_-$, there is a commutative diagram
    \[
    \begin{tikzcd}
        \varpi L^+N_{w\lambda^\vee}/\varpi^2 L^+N_{w\lambda^\vee} \times \prod_{ \substack{\alpha \in \Phi^+ \\
        \anbr{w^{-1}\alpha, \lambda}>0}} L^+N_\alpha/\varpi L^+N_\alpha  \dar["p_2"]  \rar{\sim} & \cL_w  \dar{\mathring\phi} \dar  \\
        \prod_{\substack{\alpha \in \Phi^+, \\ \anbr{w^{-1}\alpha,\lambda} > 0}} \bar N_\alpha^{\pf} \rar{\sim} &  (\bar N w \bar P_\lambda/\bar P_\lambda)^{\pf}
    \end{tikzcd}
    \]
    where $p_2$ is the projection map to the second factor, making $\cL_w \to (\bar N w \bar P_\lambda/\bar P_\lambda)^{\pf}$ an $(\AA^1)^{\pf}$-bundle.
    \item If $w \lambda^\vee \in \Phi_+$, then $\cL_w = \MV_{\lambda,w\lambda}$.
\end{enumerate}
\end{lemma}
\begin{proof}
Part (2) is proven in \autoref{lem:quasi_minuscule_semiinfinite_orbit_bundle}.
For (1), we have the following pullback diagram
\[
\begin{tikzcd}
    L^+Nw \cP_\lambda/(L^+G \cap \ad(\varpi^\lambda)L^+G) \rar \ar[d] & L^+G/(L^+G \cap \ad(\varpi^\lambda) L^+G) \simeq \Gr_\lambda \ar[d] \\
    (\bar N w \bar P_\lambda/\bar P_\lambda)^{\pf} \rar[hook]  & (\bar G/\bar P_\lambda)^{\pf}
\end{tikzcd}
\]
where $\cP_\lambda$ is the parahoric defined in \autoref{def:p_lambda}. Let $\cH_\lambda := L^+G \cap \ad(\varpi^\lambda)L^+G$. By the Iwahori decomposition for $\cP_\lambda$ and $\cH_\lambda$ (spelled out in more detail below in \autoref{cohomology_on_quasi_minuscule_weyl_orbit_intersection}),
\[
    (\varpi L^+N_{\lambda^\vee}) \cdot \cH_\lambda = \cP_\lambda,
\]
so we can rewrite the top left quotient as
\[
    L^+N w (\varpi L^+N_{\lambda^\vee})\cH_\lambda/ \cH_\lambda = L^+N(\varpi L^+N_{w\lambda^\vee})w\cH_\lambda/ \cH_\lambda
\]
We can define an action of $L^+N \times (\varpi L^+N_{w\lambda^\vee})$ on $L^+G/\cH_\lambda$ using the formula
    \[ (m,n) \cdot g\cH_\lambda = mng\cH_\lambda. \]
The stabilizer of $w\cH_\lambda$ is given by pairs $(m,n)$ such that $mn \in \ad(w)\cH_\lambda$, which translates to
\[
    \ad(\varpi^{-w\lambda})(mn) \in L^+G.
\]
So by decomposing into root spaces, we find that
\begin{align*}
    \cL_w &\simeq \smbr{\prod_{\alpha \in \Phi^+} L^+N_\alpha/\varpi^{\max \smbr{ \anbr{w^{-1}\alpha, \lambda},0}} L^+N_\alpha} \times \varpi L^+N_{w\lambda^\vee}/\varpi^{\max \smbr{ 1, \anbr{w \lambda , w\lambda^\vee}}} L^+N_{w\lambda^\vee} \\
    &= \smbr{\prod_{ \substack{\alpha \in \Phi^+ \\
    \anbr{w^{-1}\alpha, \lambda}>0}} L^+N_\alpha/\varpi L^+N_\alpha} \times \varpi L^+N_{w\lambda^\vee}/\varpi^2 L^+N_{w\lambda^\vee}
\end{align*}
The desired diagram follows.
\end{proof}

\subsection{The case of \texorpdfstring{$\nu=\lambda$}{nu is lambda}}
\label{dominant_case_equal_cocharacters}
Suppose $\mu \in X_*(T)_+$ is dominant. In this section we treat the case where $\nu = \lambda$. Since the highest weight representation $V^{\lambda+\mu}$ appears with multiplicity one in $V^\lambda \otimes V^\mu$, \autoref{conj:full_CS} translates to:

\begin{proposition}
\label{p:dominant_equal_case_main_theorem}
If $\lambda,\mu \in X_*(T)_+$, then
\[
    R\Gamma_c(\MV_{\lambda,\lambda}, \cA_\lambda \otimes (h_\mu^{\lambda,\lambda})^*(\cL_\psi)) =  R\Gamma_c(\MV_{\lambda,\lambda}, \cA_\lambda) = \QQellbar[-2\anbr{\rho,\lambda}]\smbr{-\anbr{\rho,\lambda}}
\]
\end{proposition}
\begin{proof}
We first show that $h_\mu^{\lambda,\lambda}$ is trivial. Since $\MV_{\lambda,\lambda}$ and $L^{\geq-s}/L^+\GG_a$ are both pfp perfect $k$-schemes, by \autoref{lem:morphism_k_points} and \autoref{lem:weyl_orbit_intersection} it suffices to check that
    \[ h_\mu^{\lambda,\lambda}(k): N(\cO)\varpi^{\lambda}G(\cO)/G(\cO) \to \varpi^{-s}\cO/\cO \]
sends every element to $0$. But this is clear from the definition of $h$ (see \autoref{lemma:definition_of_h_map}). Therefore
    \[ (h_\mu^{\lambda,\lambda})^*\cL_\psi = \QQellbar, \]
so the first equality holds. By \cite[Prop 2.7]{Zhu14}, $R\Gamma_c(\MV_{\lambda,\lambda},\cA_\lambda)$ is concentrated in degree $2\anbr{\rho,\lambda}$. The number of irreducible components of $\MV_{\lambda, \lambda}$ is equal to the dimension of the weight space $\lambda$ in the highest weight representation $V_\lambda$, \cite[Prop. 2.8]{Zhu14}. These irreducible components form a basis of the cohomology, \cite[Prop 2.9]{Zhu14}. Finally, observe that the the weight space of $\lambda$ in $V^\lambda$ is 1-dimensional.
\end{proof}

\subsection{The case of \texorpdfstring{$\nu=w\lambda$}{nu is wlambda}}
\label{dominant_case_weyl_orbit}
Now suppose $\lambda$ is dominant, $\mu$ is dominant, and $\nu = w\lambda$ for some $w \in W$. We prove \autoref{conj:full_CS} for minuscule and quasi-minuscule $\lambda$.

Recall from \autoref{def:minuscule_quasiminuscule} that $\lambda$ is minuscule if it is minimal for the Bruhat ordering on $X_*(T)_+ \setminus \set{0}$ and $\Omega(\lambda) = W \cdot \lambda$.

\begin{proposition}
\label{cohomology_on_minuscule_weyl_orbit_intersection}
If $\lambda \in X_*(T)_+$ is minuscule, then for all $w \in W$ and all $\mu \in X_*(T)_+$ such that $\mu + w\lambda \in X_*(T)_+$,
\[
    R\Gamma_c(\MV_{\lambda,w\lambda}, \cA_\lambda \otimes (h_\mu^{\lambda,w\lambda})^*(\cL_\psi)) =  R\Gamma_c(\MV_{\lambda,w\lambda}, \cA_\lambda) = \QQellbar[-\anbr{2\rho,w\lambda}](-\anbr{\rho, \lambda})
\]
\end{proposition}

\begin{proof}
    The second equality follows from \cite[Proposition 2.7-2.9]{Zhu14}. The first equality will follow if we can show that $h_\mu^{\lambda,w\lambda}: \MV_{\lambda,w\lambda} \to L^{\geq-s}\GG_a/L^+\GG_a$ factors through the identity section $\Spec k \to L^{\geq-s}\GG_a/L^+\GG_a$.

    Since $\lambda$ is minuscule, by \cite[Corollary 1.24]{Zhu14} the reduction map gives an isomorphism
        \[ \Gr_{\le \lambda} = \Gr_\lambda \xra\sim (\bar G/\bar P_\lambda)^{\pf} \]
    which restricts to an isomorphism
        \[ S_{w\lambda} \cap \Gr_\lambda \xra\sim (\bar N w \bar P_\lambda/\bar P_\lambda)^{\pf} \]
    So any $k$-point in $S_{w\lambda} \cap \Gr_\lambda$ can be written in the form
    \[
        nwp \varpi^\lambda G(\cO) \text{ for } n \in N(\cO), p \in \cP_\lambda(\cO).
    \]
    By definition of $\cP_\lambda(\cO)$ and since $\lambda$ is minuscule, one checks that $\ad(\varpi^{-\lambda})\cP_\lambda(\cO) \subseteq G(\cO)$, so
        \[ nwp\varpi^\lambda G(\cO) = nw\varpi^\lambda G(\cO). \]
    Also, as $\ad(w) \varpi^\lambda = \varpi^{w \lambda}$, we may re-express this as
    \[
        nw\varpi^{\lambda} G(\cO) = n\varpi^{w\lambda} G(\cO)
    \]
    So since $n \in N(\cO)$, $h^{\lambda, w\lambda}_\mu$ maps $k$-points to the identity section. By \autoref{lem:morphism_k_points}, it factors through the identity section.
\end{proof}

Next we treat the case where $\lambda$ is quasi-minuscule. Recall from \autoref{def:minuscule_quasiminuscule} that $\lambda$ is quasi-minuscule if it is minimal for the Bruhat ordering on $X_*(T)_+ \setminus \set{0}$ and $\Omega(\lambda) = W\lambda \cup \set{0}$.

\begin{proposition}
\label{cohomology_on_quasi_minuscule_weyl_orbit_intersection}
If $\lambda \in X_*(T)_+$ is quasi-minuscule, then for all $w \in W$ and all $\mu \in X_*(T)_+$ such that $\mu + w\lambda \in X_*(T)_+$,
\[
    R\Gamma_c(\MV_{\lambda,w\lambda}, \cA_\lambda \otimes (h_\mu^{\lambda,w\lambda})^* \cL_\psi) =  R\Gamma_c(\MV_{\lambda,w\lambda}, \cA_\lambda) = \QQellbar[-\anbr{2\rho,w\lambda}](-\anbr{\rho, w\lambda})
\]
\end{proposition}
\begin{proof}
As before, the second equality follows from \cite[Proposition 2.7-2.9]{Zhu14}. For the first equality we again show that $h_\mu^{\lambda,w\lambda}$ factors through the identity section. As in \autoref{cohomology_on_minuscule_weyl_orbit_intersection}, any $k$-point of $S_{w\lambda} \cap \Gr_\lambda$ can be written in the form $nwp \varpi^\lambda G(\cO)$ for $n \in N(\cO)$ and $p \in \cP_\lambda(\cO)$.

We claim that we may write
    \[ p = N_{\lambda^\vee}(\varpi x)\wt{p} \]
for some $x \in \cO$ and where $\wt{p}$ is such that $\ad(\varpi^{-\lambda})(\wt{p}) \in G(\cO)$. The
Iwahori decomposition says
    \[ \cP_\lambda(\cO) = (N(\cO) \cap \cP_\lambda(\cO))(\Norm_{G(F)}(T(F)) \cap \cP_\lambda(\cO))(N^-(\cO) \cap \cP_\lambda(\cO)). \]
Moreover, there is a decomposition
\[
    N(\cO) \cap \cP_\lambda(\cO) \xra\sim N_{\lambda^\vee}(\varpi\cO) \times \prod_{\alpha \in \Phi_+ \setminus \set{\lambda^\vee}} N_\alpha(\varpi^{c_\alpha}\cO), \quad c_\alpha = \begin{cases} 0 & \anbr{\alpha,\lambda} \leq 0 \\ 1 & \anbr{\alpha,\lambda} > 0 \end{cases}.
\]
We single out the first factor on the right side of the above isomorphism and show that the other factors in both decompositions are preserved by $\ad (\varpi^{-\lambda})$.

Since $\lambda$ is dominant, $\ad(\varpi^{-\lambda})N^-(\cO) \subset G(\cO)$. Furthermore one can show using Bruhat--Tits theory that
    \[ \ad(\varpi^{-\lambda})(\Norm_{G(F)}(T(F)) \cap \cP_\lambda(\cO)) \subset G(\cO). \]
Since $\lambda$ is quasi-minuscule, $\anbr{\alpha,\lambda} = 1$ for all $\alpha \in \Phi_+ \setminus \set{\lambda^\vee}$, so $\ad(\varpi^{-\lambda})N_\alpha(\varpi^{c_\alpha}\cO) \subset G(\cO)$ for all $\alpha \in \Phi_+ \setminus \set{\lambda^\vee}$. So $p = N_{\lambda^\vee}(\varpi x)\wt{p}$ with $\ad(\varpi^{-\lambda})  \wt{p} \in G(\cO)$ as desired. We can also assume that $x \not \in \varpi\cO \setminus \set{0}$; if $x = \varpi y$ for some nonzero $y \in \cO$, then
    \[ \ad(\varpi^{-\lambda})(N_{\lambda^\vee}(\varpi x)) = \ad(\varpi^{-\lambda})(N_{\lambda^\vee}(\varpi^2 y)) = N_{\lambda^\vee}(x) \in G(\cO). \]
So in conclusion, we can write
\[
    nwp\varpi^\lambda G(\cO) = n N_{w\lambda^\vee} (\varpi x) \varpi^{w\lambda }G(\cO)
\]
for some $x \in (\cO \setminus \varpi \cO) \cup \set{0}$. If $w\lambda^\vee \in \Phi_+$, then $nN_{w\lambda^\vee}(\varpi x) \in N(\cO)$, so
\[
    h_\mu^{\lambda,w\lambda}(nN_{w\lambda^\vee}(\varpi x) \varpi^{w\lambda} G(\cO)) = h_\mu(nN_{w\lambda^\vee}(\varpi x)) = 0.
\]
If $w \lambda^\vee \in \Phi_-$, then we will show that $x = 0$. Since our $k$-point lives in $S_{w\lambda}$, we can write
\begin{equation}\label{eq:1}
    m \varpi^{w\lambda} g = nN_{w \lambda^\vee} (\varpi x) \varpi^{w\lambda} \quad m \in N(F), g \in G(\cO),
\end{equation}
and
\[
    g = \ad (\varpi^{- w \lambda}) (m^{-1}n) N_{w \lambda^\vee} (w^{1- \anbr{w \lambda, w\lambda ^\vee}} x) = \ad (\varpi^{- w \lambda}) (m^{-1}n) N_{w \lambda^\vee}(\varpi^{-1}x).
\]
since $\anbr{w \lambda , w\lambda^\vee} = \anbr{\lambda,\lambda^\vee} = 2$. Note that\footnote{To see this, one can reduce to the case of $G = \GL_n$ by picking a faithful embedding $G \hra \GL_n$, then prove it directly for $\GL_n$ by induction on the indices of the rows and column in the $n \times n$ matrix.}
    \[ G(\cO) \cap N(F)N^-(F) = N(\cO)N^-(\cO). \]
Moreover, $N(\cO) \cap N^-(\cO) = \id$ so by uniqueness of the decomposition we see that $N_{w\lambda^\vee}(\varpi^{-1}x) \in N^-(\cO)$. But $x \not \in \varpi\cO$ so $x = 0$.
\end{proof}

\subsection{The case of \texorpdfstring{$\nu <\lambda$}{nu lt lambda}}
\label{sec:unequal_cocharacters}

We now assume $\lambda, \nu \in X_*(T)_+$ are dominant cocharacters satisfying $\lambda \neq \nu$. \autoref{p:MV_nonempty}, implies that $\MV_{\lambda, \nu}$ is only nonempty if $\nu < \lambda$, so we will assume this. Our goal is to show:

\begin{proposition}
\label{p:nu_le_lambda}
\[
    R\Gamma_c(\MV_{\lambda,\nu}, \cA_\lambda \otimes (h_0^{\lambda,\nu})^* \cL_\psi) = 0.
\]
\end{proposition}

This is the most technically involved step. We give the proof here, with some ingredients proven in \autoref{sec:break_down_of_cohomology} and \autoref{sec:zero_orbit}. We first reduce the problem to studying the geometry of $\MV_{\lambda,\nu}$ for $\lambda \in M$, \autoref{def:minuscule_quasiminuscule}, by using Zhu's geometric version of the PRV (Parthasarathy--Ranga Rao--Varadarajan) conjecture.

\begin{lemma}[{\cite[Lemma 2.16]{Zhu14}}]
\label{lem:minuscule_convolution}
Given $\lambda \in X_*(T)_{+}$ there exists a sequence of cocharacters $\lambda_\bullet = (\lambda_1,\dots,\lambda_n)$ such that
\begin{enumerate}
    \item $\lambda_i \in M$ for $i = 1,\dots,n$,
    \item $\lambda \leq |\lambda_\bullet|$, and
    \item $W_{\lambda_\bullet}^\lambda \neq 0$ in the decomposition
    \[
        \cA_{\lambda_1} \star \cdots \star \cA_{\lambda_n} = \bigoplus_{\substack{\xi \in X_*(T)_+, \\ \xi \leq |\lambda_\bullet|}} \cA_\xi \otimes W_{\lambda_\bullet}^\xi.
    \]
in the Satake category $\smbr{P_{L^+G}(\Gr_G), \star}$ with its usual convolution structure. Here, the dimension of $W^\xi_{\lambda_\bullet}$ is equal to the multiplicity of $\cA_\xi$ in the convolution.
\end{enumerate}
\end{lemma}

\begin{proof}[Proof of \autoref{p:nu_le_lambda}]
Fix $\lambda \in X_*(T)_+$. By \autoref{lem:minuscule_convolution} there exists a sequence $\lambda_\bullet = (\lambda_1,\dots,\lambda_n)$ such that
\begin{align*}
    & R\Gamma_c(\MV_{|\lambda_\bullet|,\nu}, (\cA_{\lambda_1} \star \cdots \star \cA_{\lambda_n}) \otimes (h_0^{|\lambda_\bullet|,\nu})^*(\cL_\psi)) \\
    &= \bigoplus_{\substack{\xi \in X_*(T)_+, \\ \xi \leq |\lambda_\bullet|}} R\Gamma_c(\MV_{\xi,\nu}, \cA_\xi \otimes (h_0^{\xi,\nu})^*(\cL_\psi)) \otimes W_{\lambda_\bullet}^\xi
\end{align*}

So if we can show that the direct factor inclusion (for $\xi = \nu$, noting that by assumption $\nu < \lambda \leq |\lambda_\bullet|$))
\[
    R\Gamma_c(\MV_{\nu,\nu}, \cA_\nu \otimes (h_0^{\nu,\nu})^*\cL_\psi) \otimes W_{\lambda_\bullet}^\nu \to R\Gamma_c(\MV_{|\lambda_\bullet|,\nu}, (\cA_{\lambda_1} \star \cdots \star \cA_{\lambda_n}) \otimes (h_0^{|\lambda_\bullet|,\nu})^*\cL_\psi)
\]
is a quasi-isomorphism, then it follows from the above direct sum decomposition that (with $\xi = \lambda$)
    \[ R\Gamma_c(\MV_{\lambda,\nu}, \cA_\lambda \otimes (h_0^{\lambda,\nu})^*\cL_\psi) \simeq 0 \]
as desired. By \autoref{p:dominant_equal_case_main_theorem}
\[
    R\Gamma_c(\MV_{\nu,\nu}, \cA_\nu \otimes (h_0^{\nu,\nu})^*\cL_\psi) \otimes W_{\lambda_\bullet}^\nu \simeq W_{\lambda_\bullet}^\nu[-2\anbr{\rho,\nu}] \smbr{-\anbr{\rho,\nu}}
\]
so it suffices to show that
\begin{equation}
\label{eq:count_whole_cohomology}
    \dim H^i(\MV_{|\lambda_\bullet|,\nu}, (\cA_{\lambda_1} \star \cdots \star \cA_{\lambda_n}) \otimes (h_0^{|\lambda_\bullet|,\nu})^*(\cL_\psi)) = \begin{cases}
        0 & i \neq 2\anbr{\rho,\nu} \\ \dim W_{\lambda_\bullet}^\nu & i = 2\anbr{\rho,\nu}
    \end{cases}
\end{equation}

\begin{definition}
If $\nu_\bullet = (\nu_1,\dots,\nu_n) \in X_*(T)$ satisfies $\nu = |\nu_\bullet|$, then we let
    \[ \mu_0 = 0 \text{ and } \mu_i := \nu_1 + \cdots + \nu_i \text{ for } i = 1,\dots,n \]
\end{definition}

Then by \autoref{p:break_down_of_cohomology} (the main result of the next section) we have a decomposition
\[
    R\Gamma_c(\MV_{|\lambda_\bullet|,\nu}, (\cA_{\lambda_1} \star \cdots \star \cA_{\lambda_n}) \otimes (h_0^{|\lambda_\bullet|,\nu})^*(\cL_\psi))
    = \bigoplus_{|\nu_\bullet|=\nu} \bigotimes_{i=1}^n  R\Gamma_c(\MV_{\lambda_i, \nu_i}, \cA_{\lambda_i} \otimes (h^{\lambda_i,\nu_i}_{\mu_{i-1}})^* \cL_\psi)
\]
So it remains to compute the right hand side of the above equality. We may make the following two assumptions on $\nu_\bullet$.
\begin{itemize}
    \item \underline{Every $\mu_i$ is dominant}. If not, then some $\mu_{i-1}$ is non-dominant. In this case
        \[ R\Gamma_c(\MV_{\lambda_i,\nu_i},\cA_{\lambda_i} \otimes (h_{\mu_{i-1}}^{\lambda_i,\nu_i})^* \cL_\psi) = 0 \]
    by \autoref{c:nondominant_vanishing}, so the whole tensor product vanishes as well.
    \item \underline{Either $\nu_i = w\lambda_i$ for some $w \in W$, or $\nu_i = 0$}. Recall that if $\lambda_i \in M$ then $\MV_{\lambda_i,\nu}$ is nonempty if and only if $\nu \in W \cdot \lambda_i \cup \set{0}$ (noting that $\nu$ can only equal $0$ if $\lambda_i$ is quasi-minuscule).
\end{itemize}

By \autoref{cohomology_on_quasi_minuscule_weyl_orbit_intersection} and \autoref{cohomology:zero_orbit} each tensor product
\[
    \bigotimes_{i=1}^n R\Gamma_c(\MV_{\lambda_i,\nu_i}, \cA_{\lambda_i} \otimes (h_{\mu_{i-1}}^{\lambda_i,\nu_i})^*\cL_\psi)
\]
is concentrated in degree
    \[ \sum_{i=1}^n 2\anbr{\rho,\nu_i} = 2\anbr{\rho,\nu} \]
and the dimension of the tensor product is equal to (see \autoref{def:notations_for_section_zero_orbit})
    \[ \prod_{i : \nu_i = 0} |\Delta_{\lambda^\vee}^{\mu_{i-1}}|. \]
We deduce from \autoref{cor:number_of_paths}
\[
    H^i_c(\MV_{\lambda_\bullet,\nu_\bullet}, \cA_{\lambda_\bullet} \otimes h^*\cL_\psi) =
    \begin{cases}
         0 &  i \not= 2\anbr{\rho, \nu} \\
         |\crbr{\tx{dominant $\lambda_\bullet$-paths for $\nu_\bullet$} }| & i = 2 \anbr{\rho, \nu}
    \end{cases}
\]
See \autoref{paths} for a recollection of dominant paths.
Finally, in \cite[Proposition 9.4, Lemme 9.5]{Ngo98} it is shown that
\[
    \dim W^{\nu}_{\lambda_\bullet} \ge  |\crbr{\text{dominant $\lambda_\bullet$ paths from $0$ to $\nu$} }|.
\]
which, for dimension reasons, verifies \autoref{eq:count_whole_cohomology}.
\end{proof}

\subsection{Recollection on Littelmann paths}
\label{paths} In this section we recall the basics of Littelmann paths, see \cite[Section 9]{NP00} or \cite{lit94}.

\begin{de}
\label{def:paths_combinatorics}
For any sequence $\lambda_\bullet = (\lambda_1,\dots,\lambda_n) \subset M$ and any $\nu_\bullet = (\nu_1,\dots,\nu_n) \subset X_*(T)$, we call the following combinatorial data a \textit{$\lambda_\bullet$-path (for $\nu_\bullet$)}:
\begin{itemize}
    \item A sequence of vertices $\mu_1,\dots,\mu_n \in X_*(T)$ such that for all $i = 1, \ldots, n$ we have $\nu_i = \mu_i - \mu_{i-1} \in \Omega(\lambda_i)$ (where $\mu_0 = 0$).
    \item For $i = 1,\dots,n$, maps
    \[
        p_i:[0,1] \ra X_*(T) \otimes_\ZZ \RR
    \]
    satisfying the following properties:
    \begin{enumerate}
        \item if $\mu_{i-1} \neq \mu_i$,
        \[
            p_i(t) = (1-t) \mu_{i-1} + t \mu_i.
        \]
        \item if $\mu_{i-1} = \mu_i$ then
        \[
        p_i(t) =
        \begin{cases}
            \mu_{i-1} - t \alpha^\vee_i & 0 \le t\le 1/2 \\
            \mu_{i-1}+(t-1) \alpha^\vee_i & 1/2 \le t \le 1
        \end{cases}
        \]
        where $\alpha^\vee_i \in \Delta^\vee_{\lambda_i}$, i.e. $\alpha_i^\vee$ is simple coroot $W$-conjugate to $\lambda_i$.
    \end{enumerate}
\end{itemize}
If $|\nu_\bullet| = \nu$ we say that the path is a \textit{$\lambda_\bullet$-path from $0$ to $\nu$}.
\end{de}

\begin{de}
\label{def:dominant_path}
    A $\lambda_\bullet$-path is \textit{dominant} if the image of $p_i$ lies in $(X_*(T) \otimes \RR)_+$ (the dominant Weyl chamber) for $i=1,\dots, n$.
\end{de}

\begin{lem}
Fix a $\lambda_\bullet$-path. The image of $p_i$ lies in $(X_*(T) \otimes \RR)_+$ if and only if
\begin{itemize}
    \item each $\mu_i$ is dominant, and
    \item $\anbr{\alpha_i,\mu_{i-1}} \ge 1$.
\end{itemize}
\end{lem}
\begin{proof}
The condition that each $\mu_i$ is dominant is clear from the definition. Since the dominant Weyl chamber is convex, the condition on the image is thus equivalent to
\[
    \anbr{\beta, \mu_{i-1} - \alpha_i^\vee/2} \ge 0 \quad \beta \in \Delta
\]
This is equivalent to
\[
    \anbr{\beta, \mu_{i-1}} \geq \frac{1}{2}\anbr{\beta, \alpha_i^\vee} \quad \beta \in \Delta
\]
If $\beta = \alpha_i$, this translates to
\[
    \anbr{\alpha_i, \mu_{i-1}} \ge 1.
\]
If $\beta \not= \alpha_i$ then $\anbr{\beta, \alpha_i^\vee} \le 0$ by considering the Cartan matrix, so the condition is vacuous.
\end{proof}

\begin{definition}
\label{def:notations_for_section_zero_orbit}
\[
    \Delta_{\lambda^\vee} := \crbr{\alpha \in \Delta : \alpha = w\lambda^\vee \text{ for some } w \in W}
\]
denote the set of simple roots Weyl-conjugate to $\lambda^\vee$. If $\mu \in X_*(T)$ we let
\[
    \Delta^\mu_{\lambda^\vee} := \crbr{\alpha \in \Delta_{\lambda^\vee} : \anbr{\alpha, \mu} \geq 1}.
\]
\end{definition}

\begin{corollary}
\label{cor:number_of_paths}
The number of dominant $\lambda_\bullet$-paths for $\nu_\bullet$ is
    \[ \prod_{i : \nu_i = 0} |\Delta_{\lambda_i^\vee}^{\mu_{i-1}}|. \]
\end{corollary}

\newpage
\section{Breaking down the convolution}
\label{sec:break_down_of_cohomology}

The goal of this section is to prove the following statement.

\begin{proposition}
\label{p:break_down_of_cohomology} Let $\lambda_\bullet = (\lambda_1,\ldots, \lambda_n)$ be a sequence of quasi-minuscule cocharacters.
If $\nu_\bullet = (\nu_1,\dots,\nu_n)$ is any tuple of cocharacters, then
\begin{align*}
    R\Gamma_c(\MV_{|\lambda_\bullet|,\nu}, (\cA_{\lambda_1} \star \cdots \star \cA_{\lambda_n}) \otimes (h_0^{|\lambda_\bullet|,\nu})^*(\cL_\psi))
    & \simeq \bigoplus_{|\nu_\bullet|=\nu} \bigotimes_{i=1}^n R\Gamma_c(\MV_{\lambda_i, \nu_i}, \cA_{\lambda_i} \otimes (h_{\mu_{i-1}}^{\lambda_i, \nu_i})^* \cL_\psi)
\end{align*}
where $\mu_i = \sum_{k=1}^i \nu_k$.
\end{proposition}
\begin{proof}
The proof is contained in \autoref{section:proof_of_breakdown}.
\end{proof}

First we recall the convolution structure on $\Gr_G$ and its relevant subspaces.

\subsection{Convolution Grassmannian}

\label{convolution_products}

\begin{definition}\label{def:twisted_product}
If $K$ is an affine group scheme and $E \ra Y$ is a $K$-torsor, then for any $K$-space $X$ (i.e., a scheme $X$ with a left action of $K$) we can form the \textit{twisted product} $(E \ra Y) \tilde\times X$ (usually just denoted $Y \tilde\times X$ for simplicity) as the contracted product
    \[ E \times^K X \to Y \]
which is an \'etale-locally trivial fiber bundle over $Y$ whose fibers are isomorphic to $X$.
\end{definition}

For us, the most important instance of this construction is the convolution Grassmannian. Let $H$ denote a smooth affine group scheme over $\cO$, and consider the $L^+H$-torsor $LH \ra \Gr_H$. Then $X = \Gr_H$ is an $L^+H$-space, so we can construct the twisted product $\Gr_H \tilde\times \Gr_H$. Moreover, since $L^+H$ acts on $LH$ from the left, $\Gr_H \tilde\times \Gr_H$ becomes an $L^+H$-space and we can contract again to get
    \[ \Gr_H \tilde\times \Gr_H \tilde\times \Gr_H := \Gr_H \tilde\times (\Gr_H \tilde\times \Gr_H). \]
Repeating this, we can construct the $n$-fold twisted product $\Gr_H \tilde\times \cdots \tilde\times \Gr_H$, which has a moduli description
    \[ (\Gr_H \tilde\times \cdots \tilde\times \Gr_H)(R) = \set{ \cE_n \dra ^{\beta_n} \cdots \dra^{\beta_2} \cE_1 \dra^{\beta_1} \cE_0} \]
where $\cE_i$ are $H$-torsors on the disk $\Spec W_\cO(R)$, $\beta_i$ are isomorphisms on the punctured disk $\Spec W_\cO(R)[1/\varpi]$, and $\cE_0$ is the trivial $H$-torsor.

\begin{definition}
\label{def:multiplication_map}
Let
\begin{align*}
    m_j : \Gr_H \tilde{\times} \cdots \tilde{\times} \Gr_H &\to \Gr_H \\
    (\cE_n \dra^{\beta_n} \cdots \dra^{\beta_2} \cE_1 \dra^{\beta_1} \cE_0) &\mapsto (\cE_j \dra^{\beta_1 \circ \cdots \circ \beta_j} \cE_0)
\end{align*}
to be the $j$th multiplication map. We denote $m := m_n$ and refer to it as the \textit{convolution map}.
\end{definition}

Put together, these maps induce an isomorphism between the twisted product and the usual product:
    \[ (m_1,\dots,m_n): \Gr_H \tilde\times \cdots \tilde\times \Gr_H \xra\sim \Gr_H^n. \]

Now let $H = G$. Given $\lambda_\bullet = (\lambda_1,\dots,\lambda_n)$ in $X_*(T)_+ \setminus \set{0}$, the convolution map restricts to a morphism
    \[ m: \Gr_{\leq\lambda_1} \tilde\times \cdots \tilde\times \Gr_{\leq\lambda_n} \to \Gr_{\leq|\lambda_\bullet|} \]
If we restrict the $L^+G$-torsor $LG \to \Gr_G$ to $\varpi^\nu LN$, then the restricted map defines an $L^+N$-torsor $\varpi^\nu LN \to S_\nu$, which fits into the diagram
\[
    \begin{tikzcd}
        \varpi^\nu LN \rar \dar[swap,hook]{L^+N} & LG \dar{L^+G} \\
        S_\nu = \varpi^\nu \Gr_N \rar[hook] & \Gr_G
    \end{tikzcd}
\]
Moreover, since $\varpi^\nu LN = LN \varpi^\nu$ and thus admits a left action of $L^+N$, we can make the following construction.

\begin{definition}
\label{def:Snu-twisted-product}
For $\nu_\bullet = (\nu_1,\dots,\nu_n)$ any tuple in $X_*(T)$ the twisted product
    \[ S_{\nu_\bullet}: = S_{\nu_1} \tilde\times \cdots \tilde\times S_{\nu_m} \]
is given with respect to the bundles $LN \ra \varpi^{\nu_i} \Gr_N$ and the left $L^+N$-spaces $\varpi^{\nu_{i+1}} \Gr_N$.
\end{definition}

If $(\nu_1,\dots,\nu_n)$ is a tuple as in \autoref{def:Snu-twisted-product}, let
    \[ \mu_i := \nu_1 + \cdots + \nu_i \]
for all $i = 1,\dots,n$. The convolution map induces an isomorphism (c.f. \cite[Equation 2.2.2]{Zhu14})
\begin{align*}
\label{eq:splitting_of_S_parts}
    m: S_{\nu_\bullet} &\xra\sim S_{\mu_1} \times \cdots \times S_{\mu_n} \\
    (\varpi^{\nu_1} x_1 , \ldots, \varpi^{\nu_n}x_n) &\mapsto (\varpi^{\mu_1} x_1, \varpi^{\mu_2} (\varpi^{-\nu_2} x_1 \varpi^{\nu_2}) x_2, \ldots)
\end{align*}
Note further that each $\Gr_{\leq\lambda}$ is preserved by the right action of $L^+N$, and therefore the $L^+N$-torsors $\varpi^\nu LN \to S_\nu$ restrict to $L^+N$-torsors $\varpi^\nu LN \times_{S_\nu} \MV_{\lambda,\nu} \to \MV_{\lambda,\nu}$. As such, we can form the twisted product of these spaces as well. So if $\nu_1,\dots,\nu_n$ is a tuple of cocharacters, then we let
    \[ \MV_{\lambda_\bullet, \nu_\bullet} = \MV_{\lambda_1,\nu_1} \tilde \times \cdots \tilde\times \MV_{\lambda_n,\nu_n} \]
and we note that $m$ also induces an isomorphism
    \[ m: \MV_{\lambda_\bullet, \nu_\bullet} \xra\sim S_{\nu_\bullet} \cap \Gr_{\leq\lambda_\bullet}. \]
We summarize all this in a diagram:
\[
    \begin{tikzcd}
        \MV_{\lambda_\bullet,\nu_\bullet} \rar{\sim} & S_{\nu_\bullet} \cap \Gr_{\le \mu _\bullet } \rar[hook] \dar[hook] \ar[dr, phantom, "\lrcorner"] & S_{\nu_\bullet}  \rar["\simeq"] \ar[d, hook] & S_{\sigma_1} \times S_{\sigma_2} \times \cdots \times S_{\sigma_n} \ar[d, hook] \\
        & \dar[swap]{m} \Gr_{\le \mu_\bullet} \rar[hook] & \dar[swap]{m} \Gr_G \tilde\times \cdots \tilde\times \Gr_G \rar & \Gr_G \times \cdots \times \Gr_G \simeq \Gr_G^n \dlar{\pr_n} \\
        & \Gr_{\le |\mu_\bullet|} \rar[hook] & \Gr_G
    \end{tikzcd}
\]

\subsection{Twisted products of sheaves}
\label{sec:twisted-products}

We briefly recall the basic properties of equivariant sheaf theory following \cite{BernsteinLunts} for readers unfamiliar.
Suppose $H$ is an affine group scheme. Suppose $E \to Y$ is an $H$-torsor and $X$ is a left $H$-space. By inverting the right action on $E$, one obtains a left action of $H$ on $E \times X$. By assumption $H$ acts on $E$ freely, so $H$ acts on $E \times X$ freely as well. The quotient by this action is, by definition, the contracted product $Y \tilde\times X$.

Then general descent theory (for example, see \cite{BernsteinLunts}) says that there is an equivalence of categories
    \[ D^b_H(E \times X) \xra\sim D^b(Y \tilde\times X) \]
between $H$-equivariant objects in the bounded derived category of sheaves on $E \times X$, and sheaves on the quotient $Y \tilde\times X$.

Often we will be in the situation where $E$ and $Y$ both admit \textit{left} $H$-actions as well, in such that a way that $E \to Y$ is $H$-equivariant. If $\cF \in D^b_H(Y)$ and $\cG \in D^b_H(X)$, then we can form $\cF \boxtimes \cG \in D^b_H(Y \times X)$, which pulls back to
    \[ p^*(\cF \boxtimes \cG) \in D^b_{H \times H}(E \times X) \]
under the map $p: E \times X \to Y \times X$. Note $p^*(\cF \boxtimes \cG)$ is an $H \times H$-equivariant sheaf; the first copy of $H$ acts on $E$ on the left, and the second copy of $H$ acts as described above. If we descend $p^*(\cF \boxtimes \cG)$ along the quotient map $q: E \times X \to Y \tilde\times X$, we will call the resulting sheaf $\cF \tilde\boxtimes \cG$, and note that it actually lives in $D^b_H(Y \tilde\times X)$, where $H$ acts (from the left) on $Y \tilde\times X$ via the left action on $E$. By construction it satisfies
    \[ p^*(\cF \boxtimes \cG) \simeq q^*(\cF \tilde\boxtimes \cG). \]

Moreover, given a sequence of $H$-torsors $Y_1,\dots,Y_n$ such that each $Y_i$ is also a left $H$-space, we can iterate this construction. If $\cF_i \in D^b_H(Y_i)$, then iterating this construction yields
    \[ \cF_1 \tilde\boxtimes \cdots \tilde\boxtimes \cF_n \in D^b_H(Y_1 \tilde\times \cdots \tilde\times Y_n). \]

\subsection{Auxiliary torsors}
\label{auxiliary_torsors}
First, recall that
    \[ \cA_{\lambda_1} \star \cdots \star \cA_{\lambda_m} = m_!\cA_{\lambda_\bullet}, \]
where the convolution map $m$ is defined in \autoref{def:multiplication_map} and where
    \[ \cA_{\lambda_\bullet} := \cA_{\lambda_1} \tilde\boxtimes \cdots \tilde\boxtimes \cA_{\lambda_m}. \]
There is a finite\footnote{This is because $\set{\nu : \MV_{\lambda,\nu} \neq \varnothing} = \Omega(\lambda)$ is finite.} stratification
    \[ m^{-1}(\MV_{|\lambda_\bullet|,\nu}) = \bigcup_{|\nu_\bullet|=\nu} \MV_{\lambda_\bullet,\nu_\bullet} \]
We will first study the cohomology of each piece $\MV_{\lambda_\bullet,\nu_\bullet}$ individually, then glue back together the resulting cohomologies along the stratification.

In the equal characteristic setting treated in \cite{NP00}, there is a splitting
    \[ \MV_{\lambda_\bullet,\nu_\bullet} := \MV_{\lambda_1,\nu_1} \times \cdots \times \MV_{\lambda_n,\nu_n}. \]
However, it is not clear whether this is true in mixed characteristic (see also \cite[Remark 2.6(ii)]{Zhu14}). Instead we follow (and elaborate on) the approach taken in \cite[Corollary 2.17]{Zhu14}, which involves collapsing the $L^+N$-torsor over $\MV_{\lambda,\nu}$ to an $L^rN$-torsor for sufficiently large $r$ and then showing that the twisted sheaves split after pullback along the $L^rN$-torsors. Moreover, these torsors have good cohomological properties since $L^r N$ is isomorphic to the perfection of $\AA_k^r$.

\begin{definition}
\label{def:torsors_over_MV_cycles}
Let $r \in \NN_{\geq 0} \cup \set{\infty}$. We can form $L^rN$-torsors over $S_\nu$ and $\MV_{\lambda,\nu}$ using the following pullback diagram:
\[ \begin{tikzcd}
    \MV_{\lambda,\nu}^{(r)} \rar \dar{p} & S_\nu^{(r)} := \varpi^\nu LN \times^{L^+N} L^rN \dar \\
    \MV_{\lambda,\nu} \rar[hook] & S_\nu
\end{tikzcd} \]
We adopt the convention $L^{\infty}N := L^+N$. Note that $S_\nu^{(0)} = S_\nu$ and $S_\nu^{(\infty)} = \varpi^\nu LN$.
\end{definition}

We already know there is an action of $L^+N$ on $\MV_{\lambda,\nu}$. The left action of $L^+N$ on $\varpi^\nu LN$ and the trivial action on $L^rN$ induces an action of $L^+N$ on $\varpi^\nu LN \times L^rN$. Since the left and right $L^+N$-actions on $\varpi^\nu LN$ commute, this descends to an action of $L^+N$ on the presheaf quotient of $\varpi^\nu LN \times L^rN$ by the left and right $L^+N$-actions. Sheafification commutes with finite products, so this gives us an action of $L^+N$ on $\varpi^\nu LN \times^{L^+N} L^rN$.

It is moreover clear from the definitions that the maps
    \[ \MV_{\lambda,\nu} \to S_\nu \text{ and } \varpi^\nu LN \times^{L^+N} L^rN \to S_\nu \]
are both $L^+N$-equivariant, so in conclusion we obtain a left action of $L^+N$ on the fiber product $\MV_{\lambda,\nu}^{(r)}$.

\begin{lemma}
\label{lem:action_on_torsor_factors}
For $r \geq 0$, the left action of $L^+N$ on $\MV_{\lambda,\nu}^{(r)}$ factors through $L^{r'}N$ for some $r' > 0$.
\end{lemma}
\begin{proof}
For the first factor, the left action of $L^+G$ on $\Gr_{\leq \lambda}$ factors through $L^{r'}G$ for some $r > 0$, so the left $L^+N$-action on $\MV_{\lambda,\nu}$ factors through $L^{r'}N$ as well.

For the second factor, if $R$ is a perfect $k$-algebra, then after replacing $R$ by an \'etale cover, an $R$-point of $\varpi^\nu LN \times^{L^+N} L^rN$ is the equivalence class of a pair $(\varpi^\nu n, L^+N^{(r)}(R))$ for some $n \in LN(R)$. Fix such a pair, and assume that its image in $S_\nu$ lands in $\MV_{\lambda,\nu}$, since this then defines an element of $\MV_{\lambda,\nu}^{(r)}$. We want to show that there exists some large enough $r'' > r'$ such that if $h \in L^+N^{(r'')}(R)$ then
    \[ (h\varpi^\nu n, L^+N^{(r)}(R)) \sim (\varpi^\nu n, L^+N^{(r)}(R))\]
Up to replacing $R$ by a further \'etale cover, an $R$-point of $\MV_{\lambda,\nu}$ can be written as an $L^+G$-coset in $LG(R)/L^+G(R)$. Since $\varpi^\nu n L^+G(R) \in \MV_{\lambda,\nu}(R)$,  if $h \in L^+N^{(r')}(R)$ then $h$ fixes $\varpi^ \nu n L^+G(R) \in \MV_{\lambda,\nu}(R)$, so  \[ h\varpi^\nu n = \varpi^\nu n g \]  for some $g \in L^+G(R)$. In fact $g \in LN(R)$, since $g = \ad(n^{-1})\ad(\varpi^{-\nu})(h)$, so $g \in L^+N(R) = LN(R) \cap L^+G(R)$. Then
     \[ (h\varpi^\nu n, L^+N^{(r)}(R)) = (\varpi^\nu n g, L^+N^{(r)}(R)) \sim (\varpi^\nu n, gL^+N^{(r)}(R)), \]
so we are done if $g \in L^+N^{(r)}(R)$. Note that it is immediate from the definitions that $L^+N^{(r)} = L^+G^{(r)} \cap L^+N$, so we are done if we can show that $g \in L^+G^{(r)}(R)$.

Since $\varpi^\nu n L^+G(R) \in \Gr_{\leq\lambda}(R)$, there exists some $x,g' \in L^+G(R)$ and some dominant $\lambda' \leq \lambda$ such that $\varpi^\nu n = x\varpi^{\lambda'}g'$.
Thus
    \[ g = \ad((g')^{-1}) \ad(\varpi^{-\lambda'}) \ad(x^{-1})(h) \]
Since $L^+G^{(r'')}(R) \subset L^+G(R)$ is normal, $\ad(x^{-1})(h) \in L^+G^{(r'')}(R)$. By \autoref{lem:conjugating_coharacter_congruence_subgroup} below,
    \[ \ad(\varpi^{-\lambda'})(L^+G^{(r'')}) \subset L^+G^{(r)} \]
for large enough $r''$, which we also choose so that $r'' > r'$. Finally, conjugation by $(g')^{-1}$ sends the normal subgroup $L^+G^{(r)}(R) \subset L^+G(R)$ to itself, so we get that $g \in L^+G^{(r)}(R)$, as desired.
\end{proof}

\begin{lemma}
\label{lem:conjugating_coharacter_congruence_subgroup}
For any cocharacter $\nu$ and any $r \in \NN$, there exists $s \in \NN$ such that
     \[ \ad(\varpi^\nu)(L^+G^{(s)}) \subseteq L^+G^{(r)}. \]
\end{lemma}
\begin{proof}
For any perfect $k$-algebra $R$, there is a decomposition
    \[ L^+G^{(s)}(R) = L^+T^{(s)}(R) \times \prod_{\alpha \in \Phi} N_\alpha(\varpi^s W_\cO(R)), \]
so
     \[ \ad(\varpi^\nu)(L^+G^{(s)}(R)) = L^+T^{(s)}(R) \times \prod_{\alpha \in \Phi} N_\alpha(\varpi^{s+\anbr{\alpha,\nu}} W_\cO(R)), \]
so take any $s \geq \max_{\alpha \in \Phi}\set{r - \anbr{\alpha,\nu}}$.
\end{proof}

Now pick $\nu_\bullet=(\nu_1,\dots,\nu_n)$ such that $|\nu_\bullet| = \nu_1 + \cdots + \nu_n = \nu$.  By \autoref{lem:action_on_torsor_factors} we can choose integers $r_1,\dots,r_n \geq 0$ such that $r_n = 0$ and such that the action of $L^+N$ on $\prod_{k=i}^n \MV_{\lambda_k,\nu_k}^{(r_k)}$ factors through $L^{r_{i-1}}N$ for $i = 2,\dots,n$.

\begin{example}
\hfill \\
    \begin{tabular}{ccccc}
    $L^+N$ & acts on & $\MV_{\lambda_m,\nu_m}$ & via & $L^+N \thra L^{r_{m-1}}N$, \\
    $L^+N$ & acts on & $\MV_{\lambda_{m-1},\nu_{m-1}}^{(r_{m-1})} \times \MV_{\lambda_m,\nu_m}$ & via & $L^+N \thra L^{r_{m-2}}N$,
\end{tabular}
and so on.
\end{example}

\begin{lemma}\label{lem:splitting_resolution}
There are two $\prod_{i=1}^n L^{r_i} N$-torsors $p_\bullet = \prod_{i=1}^n p_{r_i}$ and $q_\bullet$:
    \[
    \begin{tikzcd}
       &   \prod_{i=1}^n \MV_{\lambda_i, \nu_i}^{(r_i)} \ar[dr, "q_\bullet"]  \ar[dl, "p_\bullet", swap ] \\
     \prod_{i=1}^n \MV_{\lambda_i,\nu_i}  && \MV_{\lambda_\bullet, \nu_\bullet}
    \end{tikzcd}
     \]
such that
    \[ q_\bullet^* \cA_{\lambda_\bullet} \cong p_1^* \cA_{\lambda_1} \boxtimes \cdots \boxtimes p_m^* \cA_{\lambda_n}. \]
\end{lemma}

\begin{proof}
The torsor $p_\bullet$ is just the product of each individual $L^{r_i}N$-torsor $p_{r_i}: \MV_{\lambda_i,\nu_i}^{(r_i)} \to \MV_{\lambda_i,\nu_i}$. If $m = 1$ then both torsors are just the identity map, so suppose $m > 1$. Since the $L^+N$-action on $\MV_{\lambda_m,\nu_m}$ factors through $L^{r_{m-1}}N$, we can form the diagram
\[
\begin{tikzcd}
    & LG \times \Gr_G \ar[rd] \ar[ld] \\
    \Gr_G \times \Gr_G & \MV_{\lambda_{m-1},\nu_{m-1}}^{(\infty)} \times \MV_{\lambda_m,\nu_m} \dar[swap]{p_\infty} \arrow[rdd,"q", bend left=10] \ar[ddl,"p",swap, bend right=10] \ar[u,hook] & \Gr_G \tilde\times \Gr_G \\
    & \MV_{\lambda_{m-1}, \nu_{m-1}}^{(r_{m-1})} \times \MV_{\lambda_m,\nu_m} \arrow[rd,swap,"q_{r_{m-1}}"] \dlar{p_{r_{m-1}} \times \id} \\
    \MV_{\lambda_{m-1},\nu_{m-1}} \times \MV_{\lambda_m,\nu_m} \ar[uu,hook] &&  \MV_{\lambda_{m-1},\nu_{m-1}} \tilde\times \MV_{\lambda_m,\nu_m} \ar[uu,hook]
\end{tikzcd}
\]
in which $q$ is an $L^+N$-torsor and $q_{r_{m-1}}$ is an $L^{r_{m-1}}N$-torsor. The morphism $p_\infty$ is just the change of fibres along the map $L^+N \to L^rN$ in the first slot and the identity in the second. Descent along $q$ gives us a uniquely defined ``external twisted product'' $\cA_{\lambda_{m-1}} \tilde\boxtimes \cA_{\lambda_m}$ on $\MV_{\lambda_{m-1},\nu_{m-1}} \tilde\times \MV_{\lambda_m,\nu_m}$ satisfying
    \[ p^*(\cA_{\lambda_{m-1}} \boxtimes \cA_{\lambda_m}) \cong q^*(\cA_{\lambda_{m-1}} \tilde\boxtimes \cA_{\lambda_m}), \]
noting that $p^*(\cA_{\lambda_{m-1}} \boxtimes \cA_{\lambda_m})$ is $L^+N$-equivariant\footnote{$\cA_{\lambda_{m-1}} \tilde\boxtimes \cA_{\lambda_m}$ is the pullback of the same object defined using the $L^+G$-torsor $LG \times \Gr_G \to \Gr_G \tilde\boxtimes \Gr_G$ along the right vertical arrow in the diagram above.}. Similarly, there is a uniquely defined external twisted product $\cL$ on $\MV_{\lambda_{m-1},\nu_{m-1}} \tilde\times \MV_{\lambda_m,\nu_m}$ satisfying
    \[ p_{r_{m-1}}^*(\cA_{\lambda_{m-1}} \boxtimes \cA_{\lambda_m}) \cong q_{r_{m-1}}^*\cL \]
But pulling back by $p_\infty$ gives $q^*\cL \cong p^*(\cA_{\lambda_{m-1}} \boxtimes \cA_{\lambda_m})$ so we must have $\cL \cong \cA_{\lambda_{m-1}} \tilde\boxtimes \cA_{\lambda_m}$ by uniqueness.

If $m > 2$, one repeats this process inductively. For example, in the next step replace $\MV_{\lambda_m,\nu_m}$ with $\MV_{\lambda_{m-1},\nu_{m-1}}^{(r_{m-1})} \times \MV_{\lambda_m,\nu_m}$, run the same argument, and chase the relevant diagram.
\end{proof}

\begin{lemma}
\label{lem:splitting_h_map_under_resolution}
Fix a sequence $\nu_\bullet = (\nu_1,\dots,\nu_m)$ of cocharacters with $|\nu_\bullet|=\nu$ and let $\mu_i = \nu_1 + \cdots + \nu_i$ for $i = 1,\dots,n$. Note $\mu_n = \nu$. The following diagram commutes:
\[
    \begin{tikzcd}
        \prod_{i=1}^n \MV_{\lambda_i,\nu_i}^{(r_i)} \rar{q_\bullet} \dar[swap]{p_\bullet}
        & \wt\prod_{i=1}^n \MV_{\lambda_i,\nu_i} \dar{m} \\
        \prod_{i=1}^n \MV_{\lambda_i,\nu_i} \dar[swap]{\prod_{i=1}^n h_{\mu_{i-1}}^{\lambda_i,\nu_i}} & \MV_{|\lambda_\bullet|,\nu} \dar{h_0^{|\lambda_\bullet|,\nu}} \\
        \prod_{i=1}^n L\GG_a/L^+\GG_a \rar{+} & L\GG_a/L^+\GG_a
    \end{tikzcd}
\]
As a direct consequence,
\[
    q_\bullet^* m^* (h_0^{|\lambda_\bullet|,\nu})^* \cL_\psi
    \simeq p_1^*(h_0^{\lambda_1,\nu_1})^*\cL_\psi \boxtimes  p_2^*(h_{\mu_1}^{\lambda_2,\nu_2})^*\cL_\psi \boxtimes \cdots \boxtimes p_n^*(h_{\mu_{n-1}}^{\lambda_n,\nu_n})^*  \cL_\psi.
\]
\end{lemma}

\begin{proof}
By \autoref{lem:morphism_k_points}, commutativity can be checked on the level of $k$-points. By \autoref{def:torsors_over_MV_cycles} and the fact that $L^+N(k) \to L^rN(k)$ is surjective, we can write a general $k$-point of the top left entry in the diagram as
\[
    \smbr{\varpi^{\nu_1}x_1, \ldots, \varpi^{\nu_n }x_n} \in \prod_{i=1} ^n \MV_{\lambda_i, \nu_i}^{(r_i)}(k) \tx{ where } x_i \in N(F) \tx{ for } i=1,\ldots, n.
\]
Passing through $q_\bullet$ and going via the equivalences (c.f. \cite[Section 2.2.1]{Zhu14})
    \[ \wt\prod_{i=1}^n \MV_{\lambda_i,\nu_i} \xra\sim S_{\nu_\bullet} \cap \Gr_{\leq \lambda_\bullet} \hra S_{\nu_\bullet} \xra\sim \prod_{i=1}^n S_{\mu_i} \]
and then applying $m$, we get (see \cite[Equation 2.2.1]{Zhu14} for this formula)
    \[ \ad(\varpi^{\mu_1})x_1 \cdots \ad(\varpi^{\mu_n})x_n \varpi^{\mu_n} \in S_\nu(k). \]
Then applying $h_0^\nu$ we get
\[
    h_0^\nu(\ad(\varpi^{\mu_1}) x_1 \cdots \ad(\varpi^{\mu_n}) x_n\varpi^{\nu}) = \sum _{i=1}^m h_{\mu_i}(x_i) = \sum_{i=1}^m h_{\mu_{i-1}}^{\nu_i} (y_i\varpi^{\nu_i}G(\cO)) = \sum_{i=1}^m (h_{\mu_{i-1}}^{\lambda_i,\nu_i} \circ p_i)(y_i\varpi^{\nu_i}).
\]
where
\[
    y_i= \ad(\varpi^{\nu_i}) x_i \in N(F) \quad i=1,\ldots,n. \qedhere
\]
The direct consequence follows from the commutativity of the diagram and the fact that $\cL_\psi$ is a multiplicative local system, so
    \[ (+)^*\cL_\psi = \cL_\psi \boxtimes \cdots \boxtimes \cL_\psi. \]
\end{proof}

\subsection{Proof of \autoref{p:break_down_of_cohomology}}
\label{section:proof_of_breakdown}

By the projection formula,
\[
    m_! ( \cA_{\lambda_\bullet} \otimes m^*(h_0^{|\lambda_\bullet|,\nu})^* \cL_\psi) \simeq (\cA_{\lambda_1} \star \cdots \star \cA_{\lambda_m}) \otimes (h_0^{|\lambda_\bullet|,\nu})^* \cL_\psi.
\]
Therefore,
\[
    R\Gamma_c(\MV_{|\lambda_\bullet|,\nu}, (\cA_{\lambda_1} \star \cdots \star \cA_{\lambda_m}) \otimes (h_0^{|\lambda_\bullet|,\nu})^* \cL_\psi) \simeq R\Gamma_c\smbr{\bigcup_{|\nu_\bullet|=\nu} \MV_{\lambda_\bullet,\nu_\bullet}, \cA_{\lambda_\bullet} \otimes m^*(h_0^{|\lambda_\bullet|,\nu})^* \cL_\psi}
\]

\begin{lemma}
\[
    R\Gamma_c\smbr{\bigcup_{|\nu_\bullet|=\nu} \MV_{\lambda_\bullet,\nu_\bullet}, \cA_{\lambda_\bullet} \otimes m^*(h_0^{|\lambda_\bullet|,\nu})^* \cL_\psi} \simeq \bigoplus_{|\nu_\bullet|=\nu} R\Gamma_c(\MV_{\lambda_\bullet, \nu_\bullet}, \cA_{\lambda_\bullet} \otimes m^*(h_0^{|\lambda_\bullet|,\nu})^* \cL_\psi).
\]
\end{lemma}
\begin{proof}
This follows immediately from \autoref{p:collapse_of_ss}, noting that the summands are concentrated in even degree by \autoref{cohomology:zero_orbit}, \autoref{cohomology_on_minuscule_weyl_orbit_intersection}, \autoref{cohomology_on_quasi_minuscule_weyl_orbit_intersection}.
\end{proof}

Next, we need a lemma relating the cohomology of an $\ell$-adic sheaf to its pullback over an affine space fibration.

\begin{lemma}
\label{lem:pullback_compactly_supported_cohomology}
Suppose $X$ and $Y$ are separated pfp perfect $k$-schemes, $r > 0$ is a positive integer, and $f: X \to Y$ is an $(\AA_k^r)^{\pf}$-fibration, i.e. there exists an \'etale cover $\set{U_i \to Y}_{i \in I}$ such that $X \times_Y U_i \to U_i$ is isomorphic to the projection $(\AA_k^r)^{\pf} \times U_i \to U_i$ for all $i \in I$.

Then if $\cF$ is an object in the bounded derived category of $\ell$-adic sheaves on $Y$, there is a canonical isomorphism
    \[ R\Gamma_c(Y, \cF) \xra\sim R\Gamma_c(X, f^*\cF)[2r](r). \]
\end{lemma}
\begin{proof}
We thank Pol van Hoften for explaining this to us. By assumption $f$ is a separated perfectly smooth morphism of pfp perfect schemes, equidimensional of relative dimension $r$. In this case we have an adjunction $(f_!,f^!)$, and we check that the counit
    \[ f_!f^! \cF \ra \cF \]
is an isomorphism by checking at the stalk at every geometric point $y: \Spec \kappa \to Y$. Consider the pullback diagram
\[ \begin{tikzcd}
    X_y \rar{g} \dar[swap]{\pi} & X \dar{f} \\
    \Spec \kappa \rar[swap]{y} & Y
\end{tikzcd} \]
Deperfecting this diagram and applying \cite[Proposition 6.7.10]{bhatt2014proetaletopologyschemes} gives us
\begin{align*}
    y^*(f_!f^!\cF) & \simeq \pi_! g^* f^! \cF \\
    & \simeq \pi_! g^* f^* \cF[2r](r) \\
    & \simeq \pi_! \pi^* \cF_y[2r](r)
\end{align*}
Now $\cF_y$ is isomorphic to a finite direct sum of shifts of $\QQellbar$, so since $\pi_!$ and $\pi^*$ are left adjoints (and thus preserve direct sums) it suffices to treat the case $\cF_y \simeq \QQellbar$. The result then follows since $\pi_! \QQellbar = \QQellbar[-2r](-r)$, since $\pi$ is isomorphic to the structure map $(\AA_\kappa^r)^{\pf} \to \Spec \kappa$.

The result follows by taking compactly supported cohomology.
\end{proof}

Now fix some $\nu_\bullet$ such that $|\nu_\bullet| = \nu$. Combining \autoref{lem:splitting_resolution} with \autoref{lem:splitting_h_map_under_resolution}, we see that
\begin{align*}
    q_\bullet^* (\cA_{\lambda_\bullet} \otimes m^*(h_0^{|\lambda_\bullet|,\nu})^* \cL_\psi)
    & \simeq q_\bullet^*\cA_{\lambda_\bullet} \otimes q_\bullet^*m^*(h_0^{|\lambda_\bullet|,\nu})^*\cL_\psi) \\
    & \simeq
    p_1^* (\cA_{\lambda_1}  \otimes h_{\mu_0}^* \cL_\psi) \boxtimes \cdots \boxtimes p_m^* (\cA_{\lambda_m} \otimes h_{\mu_{n-1}}^* \cL_\psi)
\end{align*}
(recall that $\mu_0 = 0$). Note that $q_\bullet^*: \prod_{i=1}^n \MV_{\lambda_i, \nu_i}^{(r_i)} \to \MV_{\lambda_\bullet,\nu_\bullet}$ is  and $\MV_{\lambda_i,\nu_i}^{r_i} \to \MV_{\lambda_i,\nu_i}$ both satisfy the hypotheses of
\autoref{lem:pullback_compactly_supported_cohomology}. Therefore,
\begin{align*}
    & R\Gamma_c(\MV_{\lambda_\bullet, \nu_\bullet}, \cA_{\lambda_\bullet} \otimes m^* (h_0^{|\lambda_\bullet|,\nu})^* \cL_\psi) \\
    & \simeq R\Gamma_c\left( \prod_{i=1}^m \MV_{\lambda_i,\nu_i}^{(r_i)}, q_\bullet^* \smbr{ \cA_{\lambda_\bullet} \otimes m^*(h_0^{|\lambda_\bullet|,\nu})^* \cL_\psi}  \right)\left[2 \dim N \cdot \sum_{i=1}^n r_i\right]\smbr{\dim N \cdot \sum_{i=1}^n r_i} \\
    & \simeq \bigotimes_{i=1}^n \smbr{R\Gamma_c(\MV_{\lambda_i, \nu_i}^{(r_i)}, p_i^*(\cA_{\lambda_i} \otimes (h_{\mu_{i-1}}^{\lambda_i,\nu_i})^* \cL_\psi))[2 \dim N \cdot r_i](\dim N \cdot r_i)} \\
    & \simeq \bigotimes_{i=1}^n R\Gamma_c(\MV_{\lambda_i, \nu_i}, \cA_{\lambda_i} \otimes (h_{\mu_{i-1}}^{\lambda_i, \nu_i})^* \cL_\psi).
\end{align*}

This concludes the proof.

\newpage

\section{Zero orbit}
\label{sec:zero_orbit}

Recall that if $\lambda$ is quasi-minuscule then $\MV_{\lambda,\nu}$ is non-empty if and only if $\nu = w\lambda$ for some $w \in W$, or if $\nu = 0$. In this section we consider the case where $\nu = 0$.

The goal of this section is the following theorem.

\begin{theorem}
\label{cohomology:zero_orbit}
If $\lambda \in X_*(T)_+$ is quasi-minuscule and $\mu \in X_*(T)_+$,
\[
    R\Gamma_c(\MV_{\lambda,0},\cA_\lambda \otimes (h_\mu^{\lambda,0})^*\cL_\psi) = \QQellbar^{|\Delta_{\lambda^\vee}^\mu|}.
\]
\end{theorem}

The geometry of $\MV_{\lambda,0}$ is more complicated  than that of $\MV_{\lambda,w\lambda}$ for $\lambda \in M$. Following \cite{NP00} and \cite{Zhu14} we use a resolution of singularities $\wt{\Gr}_{\le \lambda} \ra \Gr_{\le \lambda}$ in \autoref{resolution_singularity} to understand it. We will apply the decomposition theorem to this resolution to obtain \autoref{eq:decomposition_theorem} which reduces us to the cohomology computation in \autoref{p:cohomology_of_strata_S0capleGrmu}.

\subsection{Resolution of singularity}
\label{resolution_singularity}
We will use the resolution of $\MV_{\lambda,0}$ induced from the resolution
    \[ \pi: \wt{\Gr}_{\le\lambda} \to \Gr_{\le \lambda} \]
as defined in \cite[Lemma 2.12]{Zhu14}, which is similar in spirit to \cite[Ch. 9.1]{Billey2000SingularLO}. We briefly recall it here:

\begin{definition}
\label{def:parahoric_schemes}
Given $r \in [0,1]$, consider the parahoric subgroup
\[
   \cG_r(\cO) = \anbr{T(\cO), \varpi^{\lceil \anbr{r\lambda,\alpha} \rceil}N_\alpha(\cO) : \alpha \in \Phi} \subset G(F).
\]
This determines a parahoric $\cO$-group scheme $\cG_r$. Let $Q_r := L^+\cG_r \in \Ind\Sch_k$, as described in \autoref{def:standard_constructions}; this is representable by an affine group scheme.
\end{definition}

\begin{remark}
Note that $\cG_0 = G$ and $\cG_{1/4} = \cP_\lambda$ as defined above in \autoref{dominant_case_weyl_orbit}.
\end{remark}

\begin{example}
Suppose $G = \GL_2$. Then $\lambda = (1,-1)$. Then
\[
Q_0 = \begin{pmatrix}
    \cO & \cO \\
    \cO &\cO
\end{pmatrix}, \quad Q_{1/4}= \begin{pmatrix}
    \cO &  \varpi \cO \\
    \cO & \cO
\end{pmatrix}
\]
\[
Q_{1/2} =
\begin{pmatrix}
    \cO & \varpi \cO \\
    \varpi^{-1} \cO & \cO
\end{pmatrix} \quad  Q_{\frac{3}{4}} = \begin{pmatrix}
    \cO &  \varpi^2 \cO \\
    \varpi^{-1} \cO & \cO
\end{pmatrix}
Q_1 =
\begin{pmatrix}
    \cO & \varpi^2 \cO \\
    \varpi^{-2} \cO & \cO
\end{pmatrix}
\]
This is pictured via the following diagram
\[
\begin{tikzcd}[row sep=1ex, column sep=2ex]
    (G/P_\lambda)^{\pf} & & \smbr{\PP^1}^{\pf }\\
    &  Q_0  \ar[ul ] & & Q_{\frac{1}{2}} \ar[ul]  & & Q_{1} \\
    &  & Q_{\frac{1}{4}} \ar[ul] \ar[ur] & & Q_{\frac{3}{4}} \ar[ul] \ar[ur]
\end{tikzcd}
\]
where the top left diagonal sequences are quotients and the rest are inclusions.
\end{example}

The resolution of singularities is then
\begin{align*}
    \pi: \wt\Gr_{\leq\lambda} := Q_0 \times^{Q_{\frac14}} Q_{\frac12}/Q_{\frac34} &\to \Gr_{\leq \lambda} \\
    (g,g') &\mapsto gg'\varpi^\lambda.
\end{align*}
There is a natural map
\[
    \phi: \wt\Gr_{\leq\lambda} = Q_0 \times^{Q_{\frac14}} Q_{\frac12}/Q_{\frac34} \to Q_0 \times^{Q_{\frac14}} Q_{3/4}/Q_{3/4} = Q_0/Q_{\frac14} = (\bar G/\bar P_\lambda)^{\pf}.
\]
Since $Q_{\frac12}/Q_{\frac34} \simeq (\PP^1)^{\pf}$, the map $\pi$ is a $(\PP^1)^{\pf}$-bundle.

As stated in \cite{Zhu14}, this gives rise to a decomposition
\[
\begin{tikzcd}
    Q_0 \times^{Q_{\frac14}} Q_{\frac14}Q_{\frac34}/Q_{\frac34} \rar[equals] \dar[hook] & \pi^{-1}(\Gr_\lambda)\rar["\simeq"] \dar[hook] & \Gr_\lambda  \dar[hook] \\
    Q_0 \times^{Q_{\frac14}} Q_{\frac12}/Q_{\frac34} \rar[equals] & \wt\Gr_{\leq \lambda} \rar["\pi", swap] & \Gr_{\leq \lambda} \\
    Q_0 \times^{Q_{\frac14}} Q_{\frac14}s_{1,\mu^\vee}Q_{\frac34}/Q_{\frac34} \rar[equals] \uar[hook] & (\bar G/\bar P_\lambda)^{\pf} \uar[hook] \rar & \Gr_0 = * \uar[hook]
\end{tikzcd}
\]

We have the following identification.
\begin{proposition}
\label{p:lemma_for_section} The map
    \[
        \mathring\phi: \Gr_\lambda \xra{\pi^{-1}} \pi^{-1}(\Gr_\lambda) = Q_0 \times^{Q_{\frac14}} Q_{\frac14}Q_{\frac34}/Q_{\frac34} \xra{\phi} Q_0 \times^{Q_{\frac14}} * \xra\sim (\bar G/\bar P_\lambda)^{\pf}
    \]
    coincides with the reduction map. Thus, we have the following pullback,
\[
\begin{tikzcd}
    \cL_w  \rar \ar[d]  \ar[dr, phantom, "\lrcorner"] & Q_0 \times^{Q_{\frac{1}{4}}} Q_{\frac14}Q_{\frac34}/Q_{\frac34} \dar{\mathring\phi} \\
    (\bar Nw\bar P_\lambda /\bar P_\lambda)^{\pf} \rar & Q_0 \times ^{Q_{\frac14}} * \simeq \smbr{\bar{G}/\bar{P}_\lambda }^{\pf }
\end{tikzcd}
\]
where $\cL_w$ is as defined in  \autoref{structure_of_MV_Lambda_wLambda}.
\end{proposition}

After restriction to $S_0$ (which contains $\Gr_0$) we obtain
\begin{equation}
\label{equation:diagram_resolution_restrictted_semi_infinite_orbit }
\begin{tikzcd}
    \pi^{-1}(S_0 \cap \Gr_\lambda) \rar["\simeq"] \dar[d] &  \Gr_\lambda \cap S_0 \dar[d] \\
    \pi^{-1}(\MV_{\lambda,0}) \rar["\pi", swap] & \MV_{\lambda,0} \\
    (\bar G/\bar P_\lambda)^{\pf} \uar[hook]  \rar & \Gr_0  \uar[hook]
\end{tikzcd}
\end{equation}
Let
    \[ (\bar G/\bar P_\lambda)_-^{\pf} := \bigcup_{\substack{w \in W/\Stab_W(\lambda^\vee) \\ w\lambda^\vee \in \Phi_-}} (\bar N w \bar P_\lambda/\bar P_\lambda)^{\pf}
    \quad
    (\bar G/\bar P_\lambda)_+^{\pf} := \bigcup_{\substack{w \in W/\Stab_W(\lambda^\vee) \\ w\lambda^\vee \in \Phi_+}} (\bar N w \bar P_\lambda/\bar P_\lambda)^{\pf} \]

\begin{remark}
There is an extremely useful picture to have in mind to understand the geometry of $\wt\Gr_{\leq\lambda}$. The diagram below summarizes the situation.
\[
\begin{tikzcd}
    \MV_{\lambda,0} \rar[hook] & \Gr_{\leq\lambda} & \Gr_\lambda \lar[hook] \dar[equals] \\
    \pi^{-1}(\MV_{\lambda,0}) \uar[hook]{\pi} \drar[swap]{\phi} \rar[hook] & \wt\Gr_{\leq\lambda} \dar[swap]{\phi} \uar[hook]{\pi} & \pi^{-1}(\Gr_\lambda) \dlar{\mathring\phi} \lar[hook] \\
    & (\bar G /\bar P_\lambda)^{\pf}
\end{tikzcd}
\]
Here is a picture of this diagram for $G = \GL_2$.
\begin{center}
\begin{figure}[H]
    \centering
\includegraphics[scale=0.75]{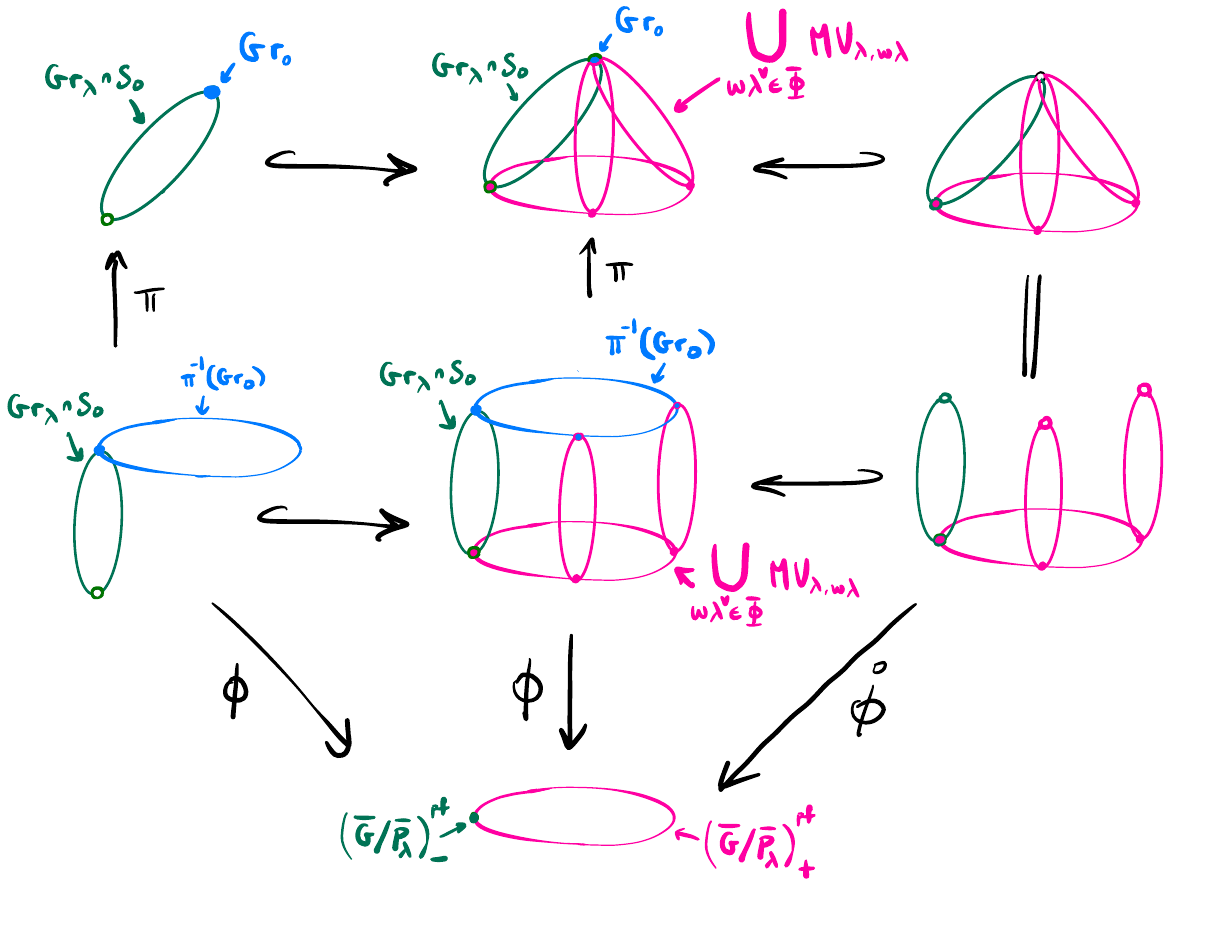}
    \caption{For $G = \GL_2$, we note that $\bar G / \bar P_\lambda \cong \PP^1$, hence we draw the base and fibers similarly. The fact that $(\bar G / \bar P_\lambda)_-^{\pf}$ is drawn as a point is also an artefact of the $\GL_2$-case, and would be a positive dimensional stratum in general.}
\end{figure}
\end{center}
\end{remark}

\begin{lem}
    $\phi(\pi^{-1}(S_0 \cap \Gr_\lambda)) \subset (\bar G/\bar P_\lambda)_-^{\pf}$.
\end{lem}

\begin{proof}
We have that
\[
    \pi^{-1}(\Gr_\lambda) = \bigsqcup_{w \in W} \pi^{-1}(\MV_{\lambda, w\lambda}) \cup \pi^{-1}(\MV_{\lambda ,0}) \simeq  \bigsqcup_{w \in W} \MV_{\lambda ,w \lambda} \cup \MV_{\lambda, 0}
\]
If $w \lambda^\vee \in \Phi_+$, we have $\MV_{\lambda, w\lambda} \simeq \phi^{-1}\smbr{\bar N w \bar P/ \bar P_\lambda}$, by \autoref{lem:quasi_minuscule_semiinfinite_orbit_bundle}, so the result follows.
\end{proof}

Following \cite{NP00}, we consider two $(\AA^1)^{\pf}$-bundles contained in $\phi^{-1}(\bar G/\bar P_\lambda)_-^{\pf}$. The first is
\[
    \cL := \bigcup_{w\lambda^\vee \in \Phi_-} \cL_w,
\]
which we already considered in \autoref{structure_of_MV_Lambda_wLambda}, and the second is $\cL^-$, see \autoref{def:line_bundle_GP}, where $\cL \cup \cL^- = \phi^{-1}(\bar G / \bar P_\lambda)_-^{\pf}$. The second bundle is necessary to compute the values of $h_\mu^{\lambda,0}$, as we do below.

The map $\phi: \wt\Gr_{\leq\lambda} \to (\bar G / \bar P_\lambda)^{\pf}$ admits a section, given by
\[
    s: (\bar G / \bar P_\lambda)^{\pf} \simeq Q_0 \times^{Q_{\frac14}} Q_{\frac34}/Q_{\frac34} \to Q_0 \times^{Q_{\frac14}} Q_{\frac12}/Q_{\frac34} = \wt\Gr_{\leq\lambda}.
\]

\begin{definition}
\label{def:line_bundle_GP}
Let
\[
    \cL^- := \phi^{-1}((\bar G/\bar P_\lambda)_-^{\pf}) \setminus s((\bar G / \bar P_\lambda)_-^{\pf})
\]
Similarly for $w \in W$ such that $w\lambda^\vee \in \Phi_-$ let
\[
    \cL^{-}_w:= \cL^-|_{(\bar{N}w\bar P_\lambda/\bar P_\lambda)^{\pf}}
\]
\end{definition}

\begin{lemma}
\label{lem:identification_of_section}
    \[ s((\bar G / \bar P_\lambda)_-^{\pf}) = \pi^{-1}\smbr{\bigcup_{w\lambda^\vee \in \Phi_-} \MV_{\lambda,w\lambda}} \]
\end{lemma}

\begin{proof}
This follows from tracing through the identifications in \autoref{structure_of_MV_Lambda_wLambda}.
\end{proof}

\begin{lem}\label{lem:twisted_decomposition_theorem}
Let $d = \anbr{2\rho,\lambda}$. With the notation as above,
\[
    \pi_*\pi^* (h_\mu^{\lambda,0})^*\cL_\psi[d] \simeq (\cA_\mu \otimes (h_\lambda^{\lambda,0})^* \cL_\psi) \oplus \cC
\]
where $\cC$ is a complex of $\QQellbar$-vector spaces supported on $\Gr_0$ satisfying
\[
H^i(\cC) =
\begin{cases}
    H^{i+d}(\bar G/\bar P_\lambda,\QQellbar) & i \ge 0 \\
    H^{i+d-2}(\bar G/\bar P_\lambda,\QQellbar) & i <0
\end{cases}
\]
\end{lem}
\begin{proof}
As in \cite[Section 2.2.2]{Zhu14} we use the decomposition theorem to obtain
    \[ \pi_*\QQellbar[d] = \cA_\lambda \oplus \cC \]
with $\cC$ having the desired cohomology. Then the projection formula gives
\begin{align*}
    \pi_* \pi^* (h_\mu^{\lambda,0})^*\cL_\psi [d] & \simeq \pi_*(\QQellbar[d] \otimes \pi^*(h_\mu^{\lambda,0})^*\cL_\psi)  \\
    & \simeq \pi_*\QQellbar[d] \otimes (h_\mu^{\lambda,0})^*\cL_\psi \\
    & \simeq \smbr{ \cA_\lambda \otimes (h_\mu^{\lambda,0})^* \cL_\psi } \oplus \smbr{ \cC \otimes (h_\mu^{\lambda,0})^*\cL_\psi} \\
    & \simeq \smbr{ \cA_\lambda \otimes (h_\mu^{\lambda,0})^* \cL_\psi } \oplus \cC. \qedhere
\end{align*}
\end{proof}

Since $\pi$ is proper (hence $\pi_* = \pi_!$) we obtain
\begin{equation}
\label{eq:decomposition_theorem}
    R\Gamma_c(\pi^{-1}(\MV_{\lambda,0}), \pi^*(h_\mu^{\lambda,0})^* \cL_\psi) [d] = R\Gamma_c(\MV_{\lambda, 0}, \cA_\lambda \otimes (h_\mu^{\lambda,0})^* \cL_\psi) \oplus \cC
\end{equation}

\subsection{Main argument}
\label{zero_orbit_cohomological_argument}
By \autoref{eq:decomposition_theorem} to prove \autoref{cohomology:zero_orbit}, it suffices to show:

\begin{p}
\label{p:cohomology_of_strata_S0capleGrmu}
Let $d = \anbr{2\rho,\lambda}$. We have
\[
    \dim H^{i+d}_c(\pi^{-1}(\MV_{\lambda,0}),\pi^*(h_\mu^{\lambda,0})^*\cL_\psi) =
    \begin{cases}
        \dim H^{i+d}(\bar G/\bar P_\lambda,\QQellbar) &\tx{if } i > 0 \\
        \dim H^{i+d-2}(\bar G/\bar P_\lambda,\QQellbar) &\tx{if } i < 0 \\
        |\Delta_{\lambda^\vee}^\mu| + |\Delta_{\lambda^\vee}|  &\tx{if } i = 0 \\
    \end{cases}
\]
\end{p}

By \autoref{lem:identification_of_section} we have the following open-closed decomposition:
\[
\begin{tikzcd}
    \cL^- \rar[hook]  & \pi^{-1}(\MV_{\lambda,0}) & \phi^{-1}((\bar G/\bar P_\lambda)_+)  \cap \pi^{-1}(\MV_{\lambda, 0}) \simeq (\bar G/\bar P_\lambda)_+^{\pf} \lar[hook]
\end{tikzcd}
\]
inducing the long exact sequence
\begin{equation}
\label{eq:long_exact_sequence_2}
    \cdots \to H^{i+d}_c(\cL^-, (h_\mu^{\lambda,0})^* \cL_\psi) \to H^{i+d}_c(\pi^{-1}(\MV_{\lambda,0}),  \pi^* (h_\mu^{\lambda,0})^* \cL_\psi) \to H^{i+d}_c( (\bar G/\bar P_\lambda)_+, \QQellbar) \to \cdots
\end{equation}

We will now analyze the three cases $i = 0$, $i < 0$, and $i > 0$. Let us first recall the dimensions of all objects of interest.

\begin{remark}[{{\cite{Bernstein1973SCHUBERTCA}}}]
\label{rem:cohomology_of_G/P}
There is a stratification
    \[ \bar G / \bar P_\lambda = \bigcup_{w \in W} \bar N w \bar P_\lambda / \bar P_\lambda \]
such that each $\bar N w \bar P_\lambda / P_\lambda$ is an affine space of dimension
\[
    \dim \bar N w \bar P_\lambda / \bar P_\lambda = \begin{cases}
        \anbr{\rho,\lambda + w\lambda} \le \frac{d}{2}-1 & w\lambda^\vee \in \Phi_-, \tx{ equality iff $w\lambda^\vee$ is opposite of simple root} \\
        \anbr{\rho,\lambda + w\lambda} - 1  \ge \frac{d}{2}& w\lambda^\vee \in \Phi_+, \tx{ equality iff $w \lambda^\vee$ is a simple root }
    \end{cases}
\]
Therefore, by \autoref{p:collapse_of_ss},
\[
    R\Gamma_c(\bar G / \bar P_\lambda, \QQellbar) = \smbr{\bigoplus_{w\lambda^\vee \in \Phi_-} \QQellbar[-2\anbr{\rho,\lambda+w\lambda}]} \oplus \smbr{\bigoplus_{w\lambda^\vee \in \Phi_+} \QQellbar[-2\anbr{\rho,\lambda+w\lambda}-1]}.
\]
In particular, we will be using the fact that
if $w\lambda^\vee \in \Phi_-$,
    \begin{center}
\begin{tabular}{c|c}
Bundle & \tx{dimension} \\
    $\cL^- $  & $d/2$  \\
    $\cL^-_w $  & $\anbr{\rho,w\lambda} + \frac{d}{2}+1 \le \frac{d}{2}$
\end{tabular}
\end{center}
\end{remark}

\subsection{Case of \texorpdfstring{$i>0$}{i gt zero}}
Since $\dim \cL^- = d/2$ by \autoref{rem:cohomology_of_G/P} and since the \'etale cohomological dimension of a finite type scheme $X$ is bounded by $2 \dim X$, \cite[VI, Theorem 1.1]{MilEC},
    \[ H^{i+d}_c(\cL^-,(h_\mu^{\lambda,0})^*\cL_\psi) = H^{i+d+1}_c(\cL^-, (h_\mu^{\lambda,0})^*\cL_\psi) = 0 \]
Further, by \autoref{rem:cohomology_of_G/P}
    \[ H^{i+d}_c((\bar G/\bar P_\lambda)_+,\QQellbar) = H^{i+d}_c(\bar G/\bar P_\lambda, \QQellbar) \]
whenever $i > 0$.

Thus, by \autoref{eq:long_exact_sequence_2}, the equality follows.

\subsection{Case of \texorpdfstring{$i=0$}{i is zero}}

Again since $\dim \cL^- \le d/2$,  we have
\[
    H^{d+1}_c(\cL^-, \pi^*(h_\mu^{\lambda,0})^*\cL_\psi) = 0.
\]
Note $\pi^*(h_\mu^{\lambda,0})^*\cL_\psi$ restricts to the constant sheaf $\QQellbar$ on $(\bar G/ \bar P_\lambda)^{\pf}_+ \subset \pi^{-1}(\Gr_0)$ since the map factors as (where $h_\mu^{\lambda,0}$ factors through $\Spec \bar{k}$ on this component)
\[
\begin{tikzcd}
    & \pi^{-1}(\MV_{\lambda,0})  \rar{\pi}  &  \MV_{\lambda,0} \\
    & \pi^{-1}(\Gr_0) \ar[u, hook] \ar[r]  &  \Gr_0 = * \ar[u, hook]
\end{tikzcd}
\]
Thus, as $H^{d-1}_c((\bar G/\bar P_\lambda)_+, \QQellbar) = 0$ by \autoref{rem:cohomology_of_G/P}, \autoref{eq:long_exact_sequence_2} reduces to
\[
    0 \to H_c^d(\cL^-, \pi^*(h_\mu^{\lambda,0})^*\cL_\psi) \to H_c^d(\pi^{-1}(\MV_{\lambda,0}), \pi^* (h_\mu^{\lambda,0})^*\cL_\psi ) \ra H^d_c((\bar{G}/\bar{P}_\lambda)_+, \QQellbar) \to 0
\]

We know $\dim H^d_c((\bar{G}/\bar{P}_\lambda)_+, \QQellbar) = |\Delta_{\lambda^\vee}|$ by  \autoref{rem:cohomology_of_G/P}.

By \autoref{p:collapse_of_ss} applied to the stratification
\[
    \cL^- = \bigcup_{w\lambda^\vee \in \Phi_-} \cL_w^-
\]
we get
\[
    \dim H^d_c(\cL, \pi^*(h_\mu^{\lambda,0})^*\cL_\psi) = \sum_{\substack{w \in W/\Stab_W(\lambda^\vee) \\ w\lambda^\vee \in \Phi_-}} \dim H^d_c(\cL_w, j_w^*\pi^*(h_\mu^{\lambda,0})^*\cL_\psi)  = |\Delta^\mu_{\lambda^\vee}|
\]
where the last equality follows from the proposition below, \autoref{cohomology_line_bundles}.

For this we will need the following lemma on the behavior of $h_\mu^{\lambda,0}$, which the reader is encouraged to skip on a first reading.

\begin{proposition}
\label{prop:hmu_on_pieces}
Suppose $w\lambda^\vee \in \Phi_-$.
\begin{enumerate}
    \item If $-w\lambda^\vee$ is a simple root and $\anbr{-w\lambda^\vee, \mu} = 0$, the following diagram commutes for any $b = \bar n w\bar P_\lambda \in (\bar N w \bar P_\lambda/\bar P_\lambda)^{\pf}(k)$ (where $\bar n \in N(k)$):
    \[
    \begin{tikzcd}
        F_b \dar \rar \ar[rrr, bend left=20, "f"]
        & \cL^-_w \ar[d]
        \rar  & \MV_{\lambda,0} \rar[swap]{h_\mu^{\lambda,0}} \ar[d]  & L^{\ge -1}\GG_a/ L^+\GG_a \\
        b \rar & (\bar{N}w\bar{P}_\lambda/\bar{P}_\lambda)^{\pf} \rar[hook] & (\bar{G}/\bar{P}_\lambda)^{\pf}
    \end{tikzcd}
    \]
    where $f$ is an isomorphism. In particular, the integer $s$ chosen in \autoref{lemma:definition_of_h_map} can be taken to be 1 for $\MV_{\lambda,0}$.
    \item Otherwise, the restriction of $h^{\lambda,0}_\mu$ to $\cL^-_w$ is trivial.
\end{enumerate}
\end{proposition}

\begin{proof}
We prove both points simultaneously. To compute $f$ we consider its restriction to the $(\GG_m)^{\pf}$-bundle
\[
    (S_0 \cap \Gr_\lambda) \cap \phi^{-1}((\bar N w \bar P_\lambda/\bar P_\lambda)^{\pf}) = \cL_w^- \cap \cL_w.
\]
In case (1), we show that $f|_{\GG_m}$ exists and is a bijection on $k$-points in the following diagram:
\[
\begin{tikzcd}
    \smbr{\GG_m}^{\pf} \simeq S_0 \cap \Gr_\lambda \cap F_b \ar[d] \rar["{f|_{\GG_m}}"] & \GG_m^{\pf} \dar[ hook]{x^{-1} \mapsto \varpi^{-1}x^{-1}} \\
    (\AA^1)^{\pf} \simeq F_b \rar["f"] & L^{\ge -1} \GG_a/ L^+\GG_a
\end{tikzcd}
\]
and since $f$ maps $\cL^-_w \setminus \cL_w$ to zero (since it maps to $\Gr_0$ under $\pi$), we deduce that $f$ is an isomorphism. For case (2) a similar argument applies to show that $f$ is the zero map.

By \autoref{identification_of_Lw}, every element $y \in (S_0 \cap \Gr_\lambda \cap F_b)(k)$ can be written in the form
\[
    y = n w N_{\lambda^\vee}(\varpi x) \varpi^\lambda G(\cO) = nN_{w\lambda^\vee}(\varpi x) \varpi^{w \lambda} G(\cO).
\]
for $x \in \cO \setminus \varpi\cO$ and $n\in N(\cO)$ lifting $\bar n$. Moreover, for a fixed lift $n$, this expression is unique up to adding an element in $\varpi \cO$ to $x$. Now let $t := -\varpi x \in \cO$ and $\alpha = w \lambda^\vee$.
For any root $\beta \in \Phi$ and $s \in \cO^\times$, the Steinberg relation (\cite[Ch3. Lemma 19]{LectureSteinbergChevalley}) says that
\[
    \beta^\vee(s) w_\beta = N_\beta(s)N_{-\beta}(-s^{-1})N_\beta(s)
\]
where $w_\beta$ is a lift of the simple Weyl element $s_\beta \in W$. Now as  $N_\beta(s)w_\beta^{-1} \in G(\cO)$ we deduce that
\[
    nN_{\alpha}(-t)\alpha^\vee(t) G(\cO) = nN_{-\alpha}(\varpi^{-1}x^{-1}) G(\cO).
\]
Since $\mu$ is dominant and $n \in N(\cO)$ we have $h_\mu(n) = 0$. If $-\alpha$ is not a simple root the map $h$ kills $N_{-\alpha}$. If $\anbr{-\alpha, \mu}>0$, then $\varpi^{\anbr{-\alpha,\mu}-1}x^{-1} \in L^+\GG_a$, and $h$ is trivial on $L^+G$. Therefore,
\[
    h_\mu^{\lambda,0}(y) = h_\mu(N_{-\alpha}(\varpi^{-1}x^{-1})) = h(N_{-\alpha}(\varpi^{\anbr{-\alpha,\mu}-1}x^{-1})) = \begin{cases}
        0 & \tx{if }\alpha \notin \Delta \tx{ or }\anbr{-\alpha, \mu}>0 \\
        \varpi^{-1}x^{-1} & \anbr{-\alpha, \mu}=0
    \end{cases}
\]
Case (2) immediately follows. If $-\alpha$ is a simple root and $\anbr{-\alpha,\mu} = 0$ then it is clear that $f|_{\GG_m}$ exists and is is a bijection, which proves case (1).
\end{proof}

\begin{p}
\label{cohomology_line_bundles}
Suppose $w\lambda^\vee \in \Phi_-$.
    \[
\dim H^d_c\smbr{\cL^-_w, j_w^*\pi^*(h_\mu^{\lambda,0})^*\cL_\psi} =
\begin{cases}
    1 & \anbr{-w\lambda^{\vee} , \mu} >0   \text{ and } -w\lambda^\vee \in \Delta \\
    0 & \tx{ otherwise }
\end{cases}
\]
\end{p}

\begin{proof} If $-w\lambda^\vee$ is not a simple root then $H^d_c(\cL_w^- , j_w^*\pi^*(h_\mu^{\lambda,0})^*\cL_\psi)$ vanishes, as $\dim \cL^-_w < \frac{d}{2}$. Suppose $-w \lambda^\vee$ is a simple root.
\begin{itemize}
    \item If $\anbr{-w\lambda^\vee, \mu}>0$ then the map $h_\mu^{\lambda,0}$ is trivial by \autoref{prop:hmu_on_pieces}, so
    \[ j_w^*\pi^*(h_\mu^{\lambda,0})^*\cL_\psi = \QQellbar \]
    then $\dim \cL^-_w = \frac{d}{2}$. By Poincar\'e duality and the fact that $\cL_w^-$ (being an $(\AA^1)^{\pf}$-fibration over the perfection of an affine space) is smooth and connected,
        \[ H^d_c(\cL^-_w, \QQellbar) \simeq H^0(\cL^-_w , \QQellbar) = \QQellbar. \]
    \item If $\anbr{-w\lambda^\vee, \mu} = 0$. By \autoref{prop:hmu_on_pieces}, we know the map $f$ is an isomorphism in the diagram
    \[
        \begin{tikzcd}
            F_b \ar[d] \ar[dr, phantom, "\lrcorner"]  \rar[hook,"i_b"] \ar[rrr, bend left=20, "f"]\rar[hook] & \cL^-_w \ar[d]
            \rar{j_w} & \MV_{\lambda,0} \rar{h_\mu^{\lambda,0}} & L^{\ge -1}\GG_a/ L\GG_a \\
            \crbr{b} \rar[hook] &\bar{N}w\bar{P}_\lambda/\bar{P}_\lambda
        \end{tikzcd}
    \]
    where $i^*_bj_w^*(h_\mu^{\lambda,0})^*\cL_\psi= f^*\cL_\psi$. So
    \[
        R\Gamma_c(F_b, f^*\cL_\psi) = R\Gamma_c(L^{\geq-1}\GG_a/L\GG_a, \cL_\psi) = 0,
    \]
    so we have trivial cohomology on the fibers of the affine bundle $\cL^-_w$ at $k$-points, and thus $j^*_w \pi^*(h_\mu^{\lambda,0})^* \cL_\psi$ vanishes.
\end{itemize}
\end{proof}

\subsection{Case of \texorpdfstring{$i<0$}{t}}  Set $\cF_\mu:= \pi^* (h_{\mu}^{\lambda,0})^* \cL_\psi$.

\autoref{rem:cohomology_of_G/P} implies $H^{i+d}_c((G/P)_+, \QQellbar) = 0$ for $i<0$ so we are reduced to showing that
\[
    \dim H^{i+d}_c(\cL^-, \cF_\mu) = \dim H^{i+d-2}_c((G/P)_{-}, \QQellbar).
\]
The right hand side has dimension
\[ \left|  \left\{ w \in W \,:\,
2\smbr{ \anbr{\rho,w\lambda^\vee} + \frac{d}{2}}  = i+d-2 \right\}  \right|
\]
By \autoref{p:collapse_of_ss} applied to the stratifications of $(\bar G / \bar P_\lambda)_-$ and $\cL^-$, it suffices to show that
\[
 H^{i+d-2}_c(\bar N w\bar P /\bar P_\lambda , \QQellbar)=  H_c^{i+d}(\cL^-_w,\cF_\mu)
\]
and that their cohomologies are concentrated in a single degree.

\begin{itemize}
    \item If $-w\lambda^\vee$ is simple and $\anbr{-w\lambda^\vee,\mu} = 0$, then by the Leray spectral sequence, \cite[Section 12.7]{milneLEC},
        \[
        E_2^{rs}:= H^r_c(\bar N w\bar P /\bar P_\lambda , R^s \phi_! \cF_\mu) \Rightarrow H^{r+s}_c(\cL^-_w, \cF_\mu).
        \]
    By base change, we have
    \[
        \smbr{R^s \phi_! \cF_\mu}_b \simeq H^s_c(F_b,i^*_b \cF_\mu) = 0
    \]
    for all $s$ and all $b$, where last equality follows as argued in \autoref{cohomology_line_bundles}.
    Moreover $-w\lambda^\vee$ is a simple root, so $\bar N w \bar P_\lambda / \bar P_\lambda$ is an affine space of dimension $\frac{d}{2}-1$ by \autoref{rem:cohomology_of_G/P}, and thus
        \[ H_c^{i+d}(\cL^-_w,\cF_\mu) = H_c^{i+d-2}(\bar N w\bar P /\bar P_\lambda ,\QQellbar) = 0 \]
    \item Otherwise, we know the restriction of $\cF_\mu$ to $\cL^-_w$ is the constant sheaf, by \autoref{prop:hmu_on_pieces}.
    We can use the \v{C}ech-to-cohomology spectral sequence (\cite[III, Thm. 2.17]{MilEC})
    to show
    \[
        H^{i+d}_c(\cL^-_w, \QQellbar) \simeq H^{i+d-2}_c(\bar N w \bar P_\lambda / \bar P_\lambda, \QQellbar).
    \]
    For an \'etale covering $\cU = \crbr{U_i \ra \bar N w\bar P /\bar P_\lambda } _{i \in I}$ which trivializes $\cL_w^-$, we are reduced to computing cohomology of the trivial affine bundle
        \[  U \times (\AA^1)^{\pf} \ra U \]
    By the K\"{u}nneth isomorphism we deduce that
    \[
        R\Gamma_c(U,\QQellbar) \xra{\simeq}  R\Gamma_c(U,\QQellbar)[2] \otimes R\Gamma_c(\AA^1,\QQellbar)
        \simeq R\Gamma_c(U \times \AA^1,\QQellbar)[2].
    \]
    This implies that
    \[
        R\Gamma_c(\cL_w^-, \QQellbar) \simeq R\Gamma_c(\bar N w \bar P_\lambda / \bar P_\lambda, \QQellbar)[2],
    \]
    from which the conclusion follows.
\end{itemize}

\newpage

\section{Recovering classical Casselman--Shalika}
\label{trace_function_dictionary}

In this section we show that \autoref{t:main_theorem} implies the Casselman--Shalika formula. The equivalent statement of the formula that we use is due to Frenkel--Gaitsgory--Kazhdan--Vilonen in \cite{FGKV}. We first recall their statement.

In \autoref{recollection_trace_function} we briefly recall the properties of the sheaf-function dictionary that are used below.
\begin{itemize}
    \item As $G(F_0)$ is unimodular, we fix a Haar measure $\mu$ on $G(F_0)$, such that $d\mu(G(\cO_0))=1$.  The \textit{spherical Hecke algebra} is defined as
    \[
        \smbr{ \cHk ,\star} :=\Fct_c(G(\cO_0)\bs G(F_0)/G(\cO_0), \QQellbar)
    \]
    the set of $G(\cO_0)$-bi-invariant $\QQellbar$-valued compactly supported functions, with its convolution commutative ring structure, $\star$. The ring $\cHk$ has a basis $\set{H_\lambda}_{\lambda \in X_*(T)_+}$
    where $H_\lambda$ is the image under Satake (see \cite[Section 3]{GrossSatake}) of the character for the highest-weight representation $V^\lambda$. It follows from work of Kato and Lusztig (see \cite[Proposition 5.1]{FGKV}) that
        \[ H_\lambda = (-1)^{2\anbr{\rho,\lambda}} \tr(\cA_\lambda) \]
    where $\tr(\cA_\lambda)$ is the function associated to $\cA_\lambda$ via the sheaf-function dictionary.
    \item The space of \textit{compactly supported unramified Whittaker functions} is
    \[
        \cWhit := \Fct_c(G(F_0)/G(\cO_0),\QQellbar)^{N(F), \psi}
    \]
    of right-$G(\cO)$ and left-$(N(F), \psi)$ invariant compactly supported functions. This has a basis $\crbr{\phi_\lambda}_{\lambda \in X_*(T)_+}$, such that $\phi_\lambda$ is supported on $N(F) \varpi^\lambda G(\cO)$ and by $\psi$-equivariance is uniquely determined by the value
    \[
        \phi_\lambda(\varpi^\lambda) = q^{-\anbr{\rho,\lambda}}.
    \]
\end{itemize}

We have an action of the Hecke algebra on the space of Whittaker functions, $ \cWhit \circlearrowleft \cHk$ given by
\begin{equation}
\label{convolution}
    f \star h  (g):= \int _{G(F)} f(x^{-1} \cdot g )h(x) \, \mathrm{d}\mu(x)  \quad h \in \cHk, f \in \cWhit
\end{equation}
The Casselman-Shalika formula, as explained in \cite[Chapter 5]{FGKV}, describes the action $\cWhit \circlearrowleft \cHk$ with respect to the two bases:
\begin{equation}
\label{eq:hecke_action_classical}
\phi_0 \star H_\lambda = \phi_\lambda
\end{equation}

In the equal characteristic case, this was proven in \cite{FGV99} and \cite{Ngo98}.\footnote{ Note that the equation $\phi_\lambda = \phi_0 \star H_\lambda$, allows us to compute $\phi_\nu \star H_\lambda$ for all $\nu \in X_*(T)_+$. See \cite[Sec. 1.1.8]{FGV99} for further discusion.}
\begin{thm}
\label{thm:hecke_action_whittaker_classical}
\autoref{eq:hecke_action_classical} holds in mixed characteristic.
\end{thm}

\begin{proof}
$\cA_\lambda$ has a Weil structure via base change, since it can be defined over the rational Witt vector affine Grassmannian (defined over $\FF_q$); see \cite[Corollary 5.3.8]{Ach-p}. $\cL_\psi$ has a Weil structure by construction (see \autoref{cor:lustzig_character_sheaves}).
Applying \autoref{thm:trace_function} with respect to the projection map $\pi:\MV_{\lambda,\nu} \ra *$ and the  sheaf $\cA_\lambda \otimes (h_0^{\lambda,\nu})^*\cL_{\psi^{-1}}$. we obtain the evaluation of the left hand side of \autoref{eq:hecke_action_classical} at $\varpi^\nu G(\cO)$:
\begin{align*}
    \tr\smbr{\pi_!  \smbr{\cA_\lambda \otimes (h_0^{\lambda,\nu})^* \cL_{\psi^{-1}}}} & = \int_\pi \tr\cA_\lambda (x) \cdot (\tr \cL_{\psi^{-1}} \circ h_0^{\lambda,\nu})(x) \\
    &= \sum_{n \varpi^\nu G(\cO)  \in \MV_{\lambda,\nu}(k)} A_\lambda(n\varpi^\nu) \psi^{-1}(n)
\end{align*}

That $\tr \cL_{\psi^{-1}} \circ h_0^{\lambda,\nu}$ coincides with $\psi^{-1}$ follows by construction (see \autoref{cor:lustzig_character_sheaves}). By \autoref{lem:comparing_integrals},
this coincides with the left hand side of  \autoref{eq:hecke_action_classical} up to a sign of $(-1)^{2 \anbr{\rho ,\lambda}}$.
On the other hand,
\[
\tr(\QQellbar [-\anbr{2\rho, \lambda}] (-\anbr{\rho, \lambda})) =
\begin{cases}
   (-1)^{-2\anbr{\rho, \lambda}} q^{-\anbr{\rho, \lambda}} & \nu=\lambda \\
    0 & \tx{ otherwise }
\end{cases}
\]
Combining these two computations and the preceding discussion yields \autoref{eq:hecke_action_classical}.
\end{proof}

\begin{lem}
\label{lem:comparing_integrals}
    Suppose that $G$ is unimodular topological group. $\mu_{G/H}$ the canonical induced quotient measure, and that $\mu(H)=1$. Let $f$ be a function which is right $H$-equivariant. Then
    \[
    \int_G f(g) \, \mathrm{d}\mu_G(g) = \int_{G/H} f(gH) \, \mathrm{d}\mu_{G/H}(gH)
    \]
\end{lem}

\begin{proof} We have by definition
\[
\int_{G/H} \smbr{\int_H f(gh) \, \mathrm{d}\mu_H(h) } \,\mathrm{d}\mu_{G/H}(gH) = \int_G f(g) \, \mathrm{d}\mu_G(g) .
\]
In particular, if $f$ is \textit{right $H$-equivariant and $\mathrm{d}\mu_H(H)=1$}, then
\[
\int_{G/H}f(g) \, \mathrm{d}\mu_H(H) \mathrm{d}\mu_{G/H}(gH) = \int_{G/H} f(g) \, \mathrm{d}\mu_{G/H} (gH).
\]
\end{proof}

\subsection{Recollection on the sheaf-function dictionary}
\label{recollection_trace_function}

If $X_0$ is a pfp perfect $\FF_q$-scheme, let $X:=\Spec k \times_{\Spec \FF_q} X_0$ and consider the relative $q$-Frobenius
    \[ \Fr_q = \id_{\Spec k} \times \Fr_{0,q}: X \to X \]
where $\Fr_{0,q}$ is the absolute $q$-Frobenius on $X_0$.

Let us recall the basic properties of the sheaf-function dictionary.

\begin{definition}
\label{def:trace_function}
If $\cF \in D^b_c(X,\QQellbar)$ equipped with a \textit{Weil structure} $\theta: \Fr^*_q \cF \xra{\simeq} \cF$
and $x \in X_0(\FF_q) = X(k)^{\Fr_q}$, define
\begin{align*}
    \tr(\cF): X_0(k) &\ra \QQellbar \\
    x &\mapsto \sum_{i \in \ZZ} (-1)^i \tr (\theta_{\bar x}, H^i(X,\cF_{\bar{x}}))
\end{align*}
\end{definition}

\begin{p}
\label{thm:trace_function}
If $f_0:X_0 \ra Y_0$ is morphism of pfp perfect $\FF_q$-schemes, let $\cF,\cF' \in D^b_c(X,\QQellbar)$, $\cG \in D^b_c(Y, \QQellbar)$ be sheaves with Weil structures. Then
\begin{enumerate}
    \item $\tr (\cF \otimes \cF') = \tr (\cF) \tr (\cF')$
    \item $tr(f^*\cF) = \tr(\cF) \circ f$
    \item If $f$ is proper,
    \[
        \tr(f_!\cF) = \int_f \tr(\cF)
    \]
    where
    \[
        \int_f \tr(\cF) (y) = \sum_{x \in X_0(\FF_{q}), x \in f^{-1}(y)} \tr(\cF)(x)
    \]
\end{enumerate}
\end{p}

\newpage

\appendix

\section{Cohomology of stratified spaces}
\label{sec:cohomology_stratified_space}

We record a basic result about the compactly supported cohomology of complexes of sheaves on stratified schemes that we need. We think this is fairly standard, but cannot find a reference, so we write it out here.

Let $X$ be a finitely presented $k$-scheme with a finite stratification by locally closed subschemes
    \[ X = \bigcup_{\alpha \in A} C_\alpha. \]
where $A$ is some finite partially ordered set. If $|A| = n$, pick an order-preserving bijection
    \[ \sigma: A \ra \set{1,\dots,n} \]
and let $C_p = C_{\sigma^{-1}(p)}$ for $p = 1,\dots,n$.
\begin{p}
\label{p:spectral_sequence_for_stratification}
Let $\cF \in D^b_c(X,\QQellbar)$, there is a spectral sequence\footnote{Implicitly we mean that $E_1^{p,q} = 0$ if $p \not \in \set{1,\dots,n}$.}
\[
    E_1^{p,q} = H^{p+q}_c(C_p, \cF|_{C_p}) \Rightarrow H^{p+q}_c(X,\cF)
\]
\end{p}
\begin{proof}
This is fairly straightforward (see for example \cite[Section 3]{Arapura}), but we work out the details for the reader's convenience. Let $X_p = \bigcup_{i=1}^p C_i$. Then there is a decreasing sequence
    \[ X = X_n \supset X_{n-1} \supset \cdots \supset X_1 \supset X_0 = \varnothing \]
of closed immersions such that $X_p \setminus X_{p-1} = C_p$. If we let $j_p: X \setminus X_p \hra X$ denote the open immersion, then this gives rise to an exhaustive decreasing filtration
    \[ \cdots = \cF = \cF \supset (j_1)_! j_1^* \cF \supset \cdots \supset (j_{n-1})_! j_{n-1}^* \cF \supset (j_n)_!j_n^* \cF = 0 = 0 = \cdots \]
in $D^b_\ell(X)$. The usual spectral sequence for a filtered complex is
    \[ E_1^{p,q} = H_c^{p+q}(X, (j_{p-1})_!j_{p-1}^*\cF/(j_p)_!j_p^*\cF) \Rightarrow H_c^{p+q}(X,\cF). \]
Consider the diagram
\[
\begin{tikzcd}
    & C_p = X_p \setminus X_{p-1} \dar[swap]{i_{p-1,p}} \drar{i_p} &  \\
    X \setminus X_p \rar[hook]{j_{p-1,p}} \ar[rr,bend right,"j_p"] & X \setminus X_{p-1} \rar[hook]{j_{p-1}} & X
\end{tikzcd}
\]
Observe that
    \[ (j_p)_!j_p^*\cF = (j_{p-1})_!(j_{p-1,p})_!j_{p-1,p}^* j_{p-1}^* \cF. \]
Since $(j_{p-1})_!$ is a left adjoint it preserves quotients, so
    \[ (j_{p-1})_!j_{p-1}^*\cF/(j_p)_!j_p^*\cF = (j_{p-1})_![j_{p-1}^*\cF/(j_{p-1,p})_!j_{p-1,p}^*j_{p-1}^*\cF]. \]
The usual open-closed exact sequence gives
    \[ 0 \to (j_{p-1,p})_!j_{p-1,p}^*j_{p-1}^*\cF \to j_{p-1}^*\cF \to (i_{p-1,p})_*i_{p-1,p}^*j_{p-1}^*\cF \to 0, \]
so
    \[ (j_{p-1})_!j_{p-1}^*\cF/(j_p)_!j_p^*\cF = (j_{p-1})_!(i_{p-1,p})_*i_{p-1,p}^*j_{p-1}^*\cF = (i_p)_!i_p^*\cF, \]
so we conclude by noting that by definition,
    \[ H^{p+q}_c(X, (i_p)_!i_p^*\cF) = H^{p+q}_c(C_p, \cF|_{C_p}). \qedhere \]
\end{proof}

Thus, the cohomology of a complex of sheaves on $X$ is determined by its restriction to each of the strata. If the cohomology of the strata is simple, the total cohomology is determined in a simple way:

\begin{p}
\label{p:collapse_of_ss}
Assume the notation in \autoref{p:spectral_sequence_for_stratification}. If each $R\Gamma_c(C_p, \cF|_{C_p})$ has cohomology concentrated in degrees of the same parity, and this parity of is constant as $p$ varies, then
    \[ R\Gamma_c(X,\cF) \simeq \bigoplus_{p=1}^n R\Gamma_c(C_p, \cF|_{C_p}). \]
\end{p}
\begin{proof}
Adjacent columns in $E_1^{p,q}$ are concentrated in degrees of distinct parities. Therefore the differentials vanish on every page, by considering their shape.
\end{proof}

\newpage

\printbibliography

\end{document}